\documentclass{amsart}

\usepackage{amssymb}
\usepackage{caption}
\usepackage{float}
\newcommand{\e}{\rm e}

\newcommand{\rlaw}{\stackrel {(d)}{=}}

\def\e{\hbox {\rm e}}

\def\P{{\mathbb P}}
\def\R{{\bf R}}
\def\R{{\bf R}}

\def\E{{\mathbb E}}

\def\rp{\right)}
\def\lp{\left(}

\newtheorem{theorem}{Theorem}[section]
\newtheorem{lemma}[theorem]{Lemma}
\theoremstyle{definition}
\newtheorem{definition}[theorem]{Definition}

\theoremstyle{remark}
\newtheorem{remark}[theorem]{Remark}
\numberwithin{equation}{section}
\theoremstyle{proposition}

\newtheorem{corollary}[theorem]{Corollary}

\begin{document}

\title[Probabilistic aspects of  
Jacobi theta functions]{Probabilistic aspects of  
Jacobi theta functions}


\author{Paavo Salminen}
\address{ Abo Akademi University, 
Faculty of Science and Engineering,
FIN-20500 Abo, Finland}
\email{phsalmin@abo.fi}

\author{Christophe Vignat}
\address{Tulane University, New Orleans, USA and
Universit\'{e} Paris Saclay, L.S.S., CentraleSupelec, France}
\email{cvignat@tulane.edu}



\begin{abstract}
 In this note we deduce well known modular identities for Jacobi theta functions using the spectral representations associated with the real valued Brownian motion taking values on $[-1,+1]$. We consider two cases: (i) reflection at $-1$ and $+1$, (ii) killing at $-1$ and $+1$. It is seen that these two representations give, in a sense, most compact forms of the modular theta-function identities. We study also  discrete Gaussian distributions generated by theta functions, and derive, in particular,  addition formulas for discrete Gaussian variables. 
\end{abstract}

\maketitle

\section{Introduction}
\label{sec0}

Our notation for theta functions come from Whittaker and Watson \cite{Whittaker} (see also Bellman \cite{Bellman}). Perhaps the most famous/well known Jacobi theta function is 
\begin{equation}
\theta_{3}\left(z,q\right):=\sum_{n=-\infty}^{\infty}q^{n^{2}}\cos\left(2nz\right),    
\end{equation}
where $z$ and $q$ are complex numbers. In the very heart of the theory of the theta functions are the so called modular identities, see 
\cite{Armitage}, \cite{Bellman} , \cite{BianePitmanYor},   and \cite{Couwenberg} for proofs of these identities and further references. 
Historically the first such identity  was discovered  for $\theta_3$ by Poisson in 1827  (see \cite{Bellman} p. 4). In this note we call it the third modular identity and write it as follows   
\begin{equation}
\left(-\imath\tau\right)^{\frac{1}{2}}\theta_{3}\left(z\vert\tau\right)=\exp\left(\imath\tau'\frac{z^{2}}{\pi}\right)\theta_{3}\left(z\tau'\vert\tau'\right),
\label{INTRO1}
\end{equation}
where $\theta_{3}\left(z\vert\tau\right):=\theta_{3}\left(z,q\right)$, $q=\e^{\imath\pi\tau}$ with $\tau\in \R$, and $ \tau'=-1/{\tau}$. Quoting Bellman ibid. p. 4,  identity (\ref{INTRO1}) "has amazing ramifications in the fields of algebra, number theory, geometry, and other parts of mathematics. In fact, it is not easy to find another identity of comparable significance." 

From the definition of the $\theta_3$-function it is seen that it satisfies the heat equation in one dimension. This connection is, of course, classical and is pointed out in many texts, see, e.g., \cite{Bellman} p. 16, \cite{NIST} p. 533 and \cite{BianePitmanYor} p. 437.
Brownian motion - the most important stochastic process - is also intimately connected with the heat equation. One aim of  this paper is to discuss/present a unified framework to the modular identities emerging from properties of Brownian motion on a finite interval. 

The paper is organized as follows. In the next section we define the theta functions in focus, and state the modular identities associated with these. In Section 3 two Brownian motions with different behaviors at the end points of the interval $(-1,1)$ are considered. The first one is reflected at $-1$ and $1$, and the other one is killed at these points. We present the spectral representations of the transitions densities.  In this same section we also derive the corresponding representations for the Green functions. The representations of the transitions densities are our main tools to prove the modular identities of the theta functions in Section 3. The representations of the Green functions are used in Section 4 to derive some well known Mittag-Leffler expansions. Some connections with modular identities and hitting times of Brownian motion and Bessel processes are studied in Section 5. Discrete Gaussian distributions associated with the $\theta_2$- and $\theta_3$-functions are taken up in Section 6. 

The main part of the paper consists of (new) proofs of known results. In Section 6, however,  the representation of the Kolmogorov distribution function via elliptic integrals, a convolution rule for the Gaussian $\theta_2$- and $\theta_3$-distributions  and  a stability result for the Gaussian 
$\theta_3$-distribution are, to our best knowledge, new results.

\medskip
\section{Modular theta-function identities}
\label{sec1}

We start with by recalling from \cite{NIST} p. 524  the definitions of the four Jacobi theta functions using their Fourier series expansions:
\begin{align*}
\theta_{1}\left(z,q\right)&=\sum_{n\in\mathbb{Z}}\left(-1\right)^{n-\frac{1}{2}}q^{\left(n+\frac{1}{2}\right)^{2}}\e^{\imath\left(2n+1\right)z}=2\sum_{n=0}^{\infty}\left(-1\right)^{n}q^{\left(n+\frac{1}{2}\right)^{2}}\sin\left(\left(2n+1\right)z\right),\\   
\theta_{2}\left(z,q\right)&=\sum_{n\in\mathbb{Z}}q^{\left(n+\frac{1}{2}\right)^{2}}\e^{\imath\left(2n+1\right)z}=2\sum_{n=0}^{\infty}q^{\left(n+\frac{1}{2}\right)^{2}}\cos\left(\left(2n+1\right)z\right),\\    
\theta_{3}\left(z,q\right)&=\sum_{n\in\mathbb{Z}}q^{n^{2}}\e^{\imath2nz}=1+2\sum_{n=1}^{\infty}q^{n^{2}}\cos\left(2nz\right)=\sum_{n=-\infty}^{\infty}q^{n^{2}}\cos\left(2nz\right),\\
\theta_{4}\left(z,q\right)&=\sum_{n\in\mathbb{Z}}\left(-1\right)^{n}q^{n^{2}}\e^{\imath2nz}=1+2\sum_{n=1}^{\infty}\left(-1\right)^{n}q^{n^{2}}\cos\left(2nz\right)\\&\hskip3.3cm=\sum_{n=-\infty}^{\infty}\left(-1\right)^{n}q^{n^{2}}\cos\left(2nz\right),    
\end{align*}
where $q,z\in {\bf C}$ and the {\it nome} $q$ satisfies $\vert q \vert <1$. 

As a consequence of Jacobi's triple product identity, the Jacobi theta functions have the following infinite product representations \cite[(10.7.7)]{Andrews}:
\begin{equation}
\nonumber
\theta_{1}\left(z,q\right)=
2q^{\frac{1}{4}}
\sin{z}
\prod_{n=1}^{\infty} \left(1-q^{2n}\right) \left(1-2q^{2n}\cos{\left(2z\right)}+q^{4n}\right),
\label{eq:theta1 infinite product}
\end{equation}
\begin{equation}
\nonumber
\theta_{2}\left(z,q\right)=
2q^{\frac{1}{4}}
\cos{z}
\prod_{n=1}^{\infty} \left(1-q^{2n}\right)\left(1+2q^{2n}\cos{\left(2z\right)}+q^{4n}\right),
\label{eq:theta2 infinite product}
\end{equation}
\begin{equation}
\nonumber
\theta_{3}\left(z,q\right)=\prod_{n=1}^{\infty} \left(1-q^{2n}\right)\left(1+2q^{2n-1}\cos{\left(2z\right)}+q^{4n-2}\right),
\label{eq:theta3 infinite product}
\end{equation}
\begin{equation}
\nonumber
\theta_{4}\left(z,q\right)=\prod_{n=1}^{\infty} \left(1-q^{2n}\right)\left(1-2q^{2n-1}\cos{\left(2z\right)}+q^{4n-2}\right).
\label{eq:theta4 infinite product}    
\end{equation}
Notice that straightforward calculations yield
\begin{equation}
\label{C2}
\theta_{4}\left(0,q\right)=\prod_{n=1}^{\infty}\frac{ \left(1-q^{n}\right)}{\left(1+q^{n}\right)}.
\end{equation}
Recall that $h\mapsto \theta_{4}(0,\exp(-2h^2)), h>0,$ is the Kolmogorov distribution function, see Chung \cite{Chung82} and the discussion around formula (\ref{Br2}) below. 

 To introduce the modular  identities (also called Jacobi's imaginary  transformation rules), let us define, with $\imath$ the imaginary unit, the \textit{lattice parameter} $\imath \tau$, with $\tau $ such that $\Im \tau >0$, and its conjugate $\imath \tau'$ by
\[
q=\e^{\imath\pi\tau},\,\,\tau'=-\frac{1}{\tau},
\]
and write, e.g., $\theta_{1}\left(z|\tau\right)=\theta_{1}\left(z,q\right)$. 
The four modular  identities, see  \cite[20.7.30-33]{NIST} p. 531, which are of main interest in our study, are the following:  

\medskip
\noindent 
{\bf 1. First modular identity}:
\begin{equation}
\left(-\imath\tau\right)^{\frac{1}{2}}\theta_{1}\left(z\vert\tau\right)=-\imath\exp\left(\imath\tau'\frac{z^{2}}{\pi}\right)\theta_{1}\left(z\tau'\vert\tau'\right).
\label{eq:theta1 modular}
\end{equation}
Putting $q=\e^{\imath\pi\tau}$ and $\tau=\imath t$ (\ref{eq:theta1 modular}) produces the form  
\begin{align}
\label{eq:1}
\sqrt{t}\sum_{n=0}^\infty\left(-1\right)^{n}&\e^{-\pi t\left(n+\frac{1}{2}\right)^{2}}\sin\left(\left(2n+1\right)z\right)\\
\nonumber
&=\e^{-\frac{z^{2}}{\pi t}}\sum_{n=0}^\infty\left(-1\right)^{n}\e^{-\frac{\pi}{t}\left(n+\frac{1}{2}\right)^{2}}\sinh\left(\left(2n+1\right)\frac{z}{t}\right).
\end{align}

\noindent 
{\bf 2. Second modular identity}:
\begin{equation}
\left(-\imath\tau\right)^{\frac{1}{2}}\theta_{2}\left(z\vert\tau\right)=\exp\left(\imath\tau'\frac{z^{2}}{\pi}\right)\theta_{4}\left(z\tau'\vert\tau'\right).
\label{eq:theta2 modular}
\end{equation}
Putting $q=\e^{\imath\pi\tau}$ and $\tau=\imath t$ (\ref{eq:theta2 modular}) produces the form 
\begin{align}
\label{eq:21}
\sqrt{t}\sum_{n\in\mathbb{Z}}\e^{-\pi t\left(n+\frac{1}{2}\right)^{2}}&\cos\left(\left(2n+1\right)z\right)
=\e^{-\frac{z^{2}}{\pi t}}\sum_{n\in\mathbb{Z}}\left(-1\right)^{n}\e^{-\frac{\pi n^{2}}{t}}
\cosh\left(2n\frac{z}{t}\right)\\
\nonumber
&\hskip2.4cm=\sum_{n\in\mathbb{Z}}\left(-1\right)^{n}\e^{-\frac{(z+\pi n)^{2}}{\pi t}}
.
\end{align}
\noindent
{\bf 3. Third modular identity}:
\begin{equation}
\left(-\imath\tau\right)^{\frac{1}{2}}\theta_{3}\left(z\vert\tau\right)=\exp\left(\imath\tau'\frac{z^{2}}{\pi}\right)\theta_{3}\left(z\tau'\vert\tau'\right).
\label{eq:theta3 modular}
\end{equation}
Putting $q=\e^{\imath\pi\tau}$ and $\tau=\imath t$ (\ref{eq:theta3 modular}) produces the form
\begin{align}
\label{eq:31}
\sqrt{t}\sum_{n\in\mathbb{Z}}\e^{-\pi n^{2}t}\cos\left(2nz\right)&=\e^{-\frac{z^{2}}{\pi t}}\sum_{n\in\mathbb{Z}}\e^{-\frac{\pi n^{2}}{t}}\cosh\left(2n\frac{z}{t}\right)\\
\nonumber
&=\sum_{n\in\mathbb{Z}}\e^{-\frac{(z+\pi n)^{2}}{\pi t}}.
 \end{align}

\noindent 
{\bf 4. Fourth modular identity}:
\begin{equation}
\left(-\imath\tau\right)^{\frac{1}{2}}\theta_{4}\left(z\vert\tau\right)=\exp\left(\imath\tau'\frac{z^{2}}{\pi}\right)\theta_{2}\left(z\tau'\vert\tau'\right).
\label{eq:theta4 modular}
\end{equation}
Putting $q=\e^{\imath\pi\tau}$ and $\tau=\imath t$ (\ref{eq:theta4 modular}) produces the form
\begin{align}
\label{eq:41}
\sqrt{t}\sum_{n\in\mathbb{Z}}\left(-1\right)^{n}\e^{-\pi n^{2}t}\cos\left(2nz\right)&=\e^{-\frac{z^{2}}{\pi t}}\sum_{n\in\mathbb{Z}}\e^{-\frac{\pi}{t}\left(n+\frac{1}{2}\right)^{2}}\cosh\left((2n+1)\frac{z}{t}\right)\\
\nonumber
&=\sum_{n\in\mathbb{Z}}\e^{-{(z+\pi( n+\frac 12))^{2}}/{\pi t}}.
\end{align}
Notice that (\ref{eq:41}) is obtained from (\ref{eq:21}) when  $t=1/s$ and $z=\imath x/s.$


\medskip

\section{Spectral representations for  Brownian motion}
\label{sec32}
\medskip

The spectral representation formulas below  may be seen as instances of the Poisson summation formula. However, we do not focus on this aspect; instead, assuming  the expressions of the transition densities obtained by the reflection principle as known, we equate these with the associated spectral representations, also called the eigen-differential expansions. For the Poisson summation approach to Jacobi theta functions, see Bellman \cite{Bellman} and also Feller \cite{Feller} p. 342. For  more examples of spectral representations for diffusions, see Karlin and Taylor  \cite{KarlinTaylor} Section 15.13 and, for the general theory,   It\^o and McKean  \cite{ItoMckean} Section 4.11 and  Dym and McKean \cite{DymMckean} Section 5.5.

Let $X=\left(X_{t}\right)_{t\geq 0}$ denote a \textit{Brownian motion living on $[-1,+1]$ and reflected at  $-1$ and $+1$}. It is well known (and can be deduced by applying the reflection principle) that for all $x,y\in [-1,+1],$
\begin{align}
\label{ee1}
\P\left\{ X_{t}\in dy\,|\, X_0=x\right\} &  =\frac{1}{\sqrt{2\pi t}}\sum_{n=-\infty}^{\infty}\left(\e^{-\frac{\left(x-y+4n\right)^{2}}{2t}}+\e^{-\frac{\left(x+y+4n+2\right)^{2}}{2t}}\right)dy
 \\
 \nonumber 
 &
 =:p_{X}\left(t;x,y\right)2dy,
\end{align}
where $\P$ stands for a probability measure in an appropriately defined probability space. Clearly, the expression on the right-hand side of (\ref{ee1}) is well defined for all $x,y\in\R$. 

 Let, secondly,  $Y=\left(Y_{t}\right)_{t\geq  0}$ denote a \textit{Brownian motion living on $(-1,+1)$ and killed at the first exit time from the interval $(-1,+1)$}. 
 Also for $Y$ we can deduce by applying the reflection principle that for all $x,y\in [-1,+1] $
\begin{align}
\label{ee11}
%
\P\left\{ Y_{t}\in dy\,|\, Y_0=x\right\}  & =\frac{1}{\sqrt{2\pi t}}\sum_{n=-\infty}^{\infty}\left(\e^{-\frac{\left(x-y+4n\right)^{2}}{2t}}-\e^{-\frac{\left(x+y+4n+2\right)^{2}}{2t}}\right)dy
 \\
 \nonumber
 &=:p_{Y}\left(t;x,y\right)2dy.
\end{align}
 Clearly, also in this case, the expression on the right-hand side of (\ref{ee11}) is well defined for all $x,y\in\R$.

\begin{theorem}
\label{t1}
The following two identities hold for all $x,y\in\R$ and $t>0$
\begin{align}
\label{RBM}
\frac{1}{2}\frac{1}{\sqrt{2\pi t}}\sum_{n=-\infty}^{\infty}&\left(\e^{-\frac{\left(x-y+4n\right)^{2}}{2t}}+\e^{-\frac{\left(x+y+4n+2\right)^{2}}{2t}}\right)\\
\nonumber
&=\frac{1}{2}\left(\frac{1}{2}+\sum_{n=1}^{\infty}\e^{-\frac{n^{2}\pi^{2}}{2\cdot 4}t}\cos\left(\frac{n\pi\left(x+1\right)}{2}\right)\cos\left(\frac{n\pi\left(y+1\right)}{2}\right)\right)
\end{align}
and 
\begin{align}
\label{KBM}
\frac{1}{2}\frac{1}{\sqrt{2\pi t}}\sum_{n=-\infty}^{\infty}&\left(\e^{-\frac{\left(x-y+4n\right)^{2}}{2t}}-\e^{-\frac{\left(x+y+4n+2\right)^{2}}{2t}}\right)\\
\nonumber
&=\frac{1}{2}\sum_{n=1}^{\infty}\e^{-\frac{n^{2}\pi^{2}}{2\cdot 4}t}\sin\left(\frac{n\pi\left(x+1\right)}{2}\right)\sin\left(\frac{n\pi\left(y+1\right)}{2}\right).
\end{align}
\end{theorem}

\begin{proof}
The left-hand sides of formulas (\ref{RBM}) and (\ref{KBM}) represent the transition densities $p_X(t;x,y), x\in[-1,+1],$ and $p_Y(t;x,y), x\in(-1,+1)$, respectively. That the right-hand sides represent the same functions follows from the eigen-differential expansions of these densities. To make the paper more self-contained, we provide some details of these fairly standard calculations, see Ito and McKean \cite{ItoMckean} p. 149 and Dym and McKean \cite{DymMckean} p. 176 (for the Green function).  Here we consider only the distribution $p_X$, for $p_Y$ see Bellman \cite{Bellman} p.17. It holds that
\begin{align*}
p_{X}\left(t;x,y\right) =\sum_{n=0}^{\infty}\e^{-\lambda_{n}t}u_{n}\left(x\right)u_{n}\left(y\right),   
\end{align*}
where $u_n, n=0,1,2,...,$ are the eigenfunctions associated with the corresponding eigenvalues $\lambda_n, n=0,1,2,...$, that is, $u_n$ and $\lambda_n$ satisfy the ODE
$$
\frac 12 \frac{d^2}{dx^2} u_n(x)=\lambda_n u_n(x)
$$
and the boundary conditions
\[
u_{n}'\left(-1\right)=u_{n}'\left(1\right)=0.
\] 
It is easily checked that 
\[
u_{n}\left(x\right)=c_{n}\cos\left(\frac{n\pi\left(x+1\right)}{2}\right),\thinspace\thinspace n=1,2,\dots
\]
and
 $$
 \lambda_{n}=\frac{n^{2}\pi^{2}}{8},\thinspace\thinspace n=0,1,2,\dots
 $$
form the desired solution. The normalizing constants $c_{n}$ are deduced from the orthonormality
conditions, i.e.,  the family $\{u_n : n=0,1,...\} $ should satisfy
\[
\int_{-1}^{+1} u_{n}\left(x\right)u_{m}\left(x\right)2dx=\begin{cases}
0, & n\ne m,\\
1, & n=m,
\end{cases}
\]
so that
\[
c_{n}=\begin{cases}
1/2,& n=0,\\
1/{\sqrt{2}},& n=1,2,\dots\ .
\end{cases}
\]
The fact that  (\ref{RBM}) and  (\ref{KBM}) hold true for all $x,y\in\R$ follows from periodicity of the expressions and the spatial homogeneity of Brownian motion. 
\end{proof}

The Laplace transform with respect to the time parameter $t$ of a transition density of a diffusion is called the \textit{Green function} associated with this density. For $X$ and $Y$ we introduce the Green functions
\begin{equation}
\label{greenX}
 G_{\alpha}^{X}\left(x,y\right)  :=\int_{0}^{\infty}\e^{-\alpha t}p_{X}\left(t;x,y\right)dt,\quad \alpha>0,
\end{equation}
and
\begin{equation} 
\label{greenY}
 G_{\alpha}^{Y}\left(x,y\right)  :=\int_{0}^{\infty}\e^{-\alpha t}p_{Y}\left(t;x,y\right)dt,\quad \alpha\ge 0,
\end{equation}    
respectively.  The identities in the next theorem are obtained by taking the Laplace transforms on the both sides of the identities in Theorem \ref{t1}. The identity \eqref{RBMG2} is for the Green function $ G_{\alpha}^{X}$, and  \eqref{KBMG2}  for the Green function $ G_{\alpha}^{Y}$. For the expressions on the left-hand sides, see \cite{BorodinSalminen} Appendix 1 p.125-126. Recall that the Green function is a symmetric function in $x$ and $y$; hence, the following formulas  still hold in the case $x\leq y$ after switching  $x$ and $y$  on the left-hand sides.  
\begin{theorem}
\label{t11}
The following identities hold for all  for $-1\le y\le x\le 1$ and $\alpha>0$
\begin{align}
\label{RBMG2}
&
\frac{1}{\sqrt{2\alpha}\sinh\left(2\sqrt{2\alpha}\right)}\cosh\left(\left(1-x\right)\sqrt{2\alpha}\right)\cosh\left(\left(1+y\right)\sqrt{2\alpha}\right)  
\\
\nonumber
&\hskip.5cm=\frac{1}{2}\left(\frac{1}{2\alpha}+\sum_{n=1}^{\infty}\frac{1}{\alpha+\frac{n^{2}\pi^{2}}{8}}\cos\left(\frac{n\pi\left(x+1\right)}{2}\right)\cos\left(\frac{n\pi\left(y+1\right)}{2}\right)\right)
\end{align}
and 
\begin{align}
\label{KBMG2}
 &\frac{1}{\sqrt{2\alpha}\sinh\left(2\sqrt{2\alpha}\right)}\sinh\left(\left(1-x\right)\sqrt{2\alpha}\right)\sinh\left(\left(y+1\right)\sqrt{2\alpha}\right)\\
  \nonumber 
 &\hskip1cm=\frac{1}{2}\sum_{n=1}^{\infty}\frac{1}{\alpha+\frac{n^{2}\pi^{2}}{8}}\sin\left(\frac{n\pi\left(x+1\right)}{2}\right)\sin\left(\frac{n\pi\left(y+1\right)}{2}\right).
\end{align}
\end{theorem}
\section{Modular identities via Brownian motion}
\label{sec30}
\subsection{First modular identity }
\label{sec33}\medskip 
We prove here that the first modular identity can be obtained from the spectral representation formula for killed Brownian motion.
\begin{theorem} The spectral representation formula 
\eqref{KBM} implies the first modular identity \eqref{eq:1}.
\end{theorem}
\begin{proof}

Putting in (\ref{KBM}) $y=0$ and $t={8s}/{\pi}$ yields the identity
\begin{align}
\label{ss1}
\sum_{n=-\infty}^{\infty}&\left(\exp\left(-\frac{\left(x+4n\right)^{2}}{16s/\pi}\right)-\exp\left(-\frac{\left(x+4n+2\right)^{2}}{16s/\pi}\right)\right)\\
&
\nonumber
={2\sqrt{s}}\sum_{n=0}^{\infty}\left(-1\right)^{n}\exp\left(-\left(2n+1
\right)^{2}\pi s\right)\sin\left(\left(2n+1\right)\frac{\pi\left(x+1\right)}{2}\right).
\end{align}
Substituting here  $z=\frac{\pi\left(x+1\right)}{2}$ and $v=4s$, the r.h.s. of (\ref{ss1}) equals the l.h.s. of (\ref{eq:1}). Hence we consider, under this substitution, the l.h.s. of (\ref{ss1}): 
\begin{align}
\label{ss2}
\sum_{n=-\infty}^{\infty}&\left(\exp\left(-\frac{\left(x+4n\right)^{2}}{16s/\pi}\right)-\exp\left(-\frac{\left(x+4n+2\right)^{2}}{16s/\pi}\right)\right)\\
&
=\sum_{n=-\infty}^{\infty}\left(\exp\left(-\frac{\left(\frac{2z}{\pi}+4n-1\right)^{2}}{4v/\pi}\right)-\exp\left(-\frac{\left(\frac{2z}{\pi}+4n+1\right)^{2}}{4v/\pi}\right)\right) 
\nonumber
\\
\nonumber
&
=\exp\left(-\frac{z^{2}}{\pi v}\right)\sum_{n=-\infty}^{\infty}\exp\left(-\frac{\left(4n-1\right)^{2}}{4v/\pi}\right)\exp\left(-\frac{{z}\left(4n-1\right)}{v}\right)\\
\nonumber 
&\hskip1cm
-\exp\left(-\frac{z^{2}}{\pi v}\right)\sum_{n=-\infty}^{\infty}\exp\left(-\frac{\left(4n+1\right)^{2}}{4v/\pi}\right)\exp\left(-\frac{z\left(4n+1\right)}{v}\right)
\\
&
\nonumber 
=\exp\left(-\frac{z^{2}}{\pi v}\right)\sum_{n=0}^{\infty}\exp{\left(-\frac{\left(n+\frac{1}{2}\right)^{2}\pi}{v}\right)}\left(-1\right)^{n}\sinh\left(\frac{\left(2n+1\right)z}{v}\right),
\end{align}
which is equal to the r.h.s of (\ref{eq:1}).
\end{proof} 

\subsection{Second modular identity }
\label{sec35}\medskip 
We prove here that, accordingly, the second modular identity can be obtained from the spectral representation formula for killed Brownian motion.
\begin{theorem}  The spectral representation formula  \eqref{KBM} implies the second modular identity \eqref{eq:21}.
\end{theorem}
\begin{proof}
We substitute as in the proof of the first modular identity but manipulate differently. Indeed, with $y=0$ and $t=8s/\pi$  the l.h.s. of (\ref{KBM}), see (\ref{ss2}), becomes   
\begin{align*}
\frac{1}{2\sqrt{2\pi t}}\sum_{n=-\infty}^{\infty}&\left(\exp\left(-\frac{\left(x+4n\right)^{2}}{16s/\pi}\right)-\exp\left(-\frac{\left(x+4n+2\right)^{2}}{16s/\pi}\right)\right)\\
&
=\frac{\e^{-\frac{x^{2}\pi}{16s}}}{8\sqrt{s}}\left[\sum_{n=-\infty}^{\infty}\e^{-\frac{(2n)^{2}\pi}{4s}}\e^{-\frac{2nx\pi}{4s}}-\sum_{n=-\infty}^{\infty}\e^{-\frac{\left(2n+1\right)^{2}\pi}{4s}}\e^{-\frac{\left(2n+1\right)x\pi}{4s}}\right]\\
&
=\frac{\e^{-\frac{x^{2}\pi}{16s}}}{8\sqrt{s}}\Big[1+2\sum_{n=1}^{\infty}\e^{-\frac{(2n)^{2}\pi}{4s}}\cosh\left(\frac{2nx\pi}{4s}\right)\\
&\hskip3cm-2\sum_{n=1}^{\infty}\e^{-\frac{\left(2n-1\right)^{2}\pi}{4s}}\cosh\left(\frac{\left(2n-1\right)x\pi}{4s}\right)\Big]\\
&
=\frac{\e^{-\frac{x^{2}\pi}{16s}}}{8\sqrt{s}}\left[1+2\sum_{n=1}^{\infty}\e^{-\frac{n^{2}\pi}{4s}}(-1)^n\cosh\left(\frac{nx\pi}{4s}\right)\right]\\
&
=\frac {1}{8\sqrt{s}}\e^{ -\frac{x^{2}\pi}{16s}}
\sum_{n=-\infty}^{\infty}(-1)^n \e^{ -\frac{n^{2}\pi}{4s}}
\cosh\left(\frac{nx\pi}{4s}\right).
\end{align*}
Substituting similarly on the r.h.s. of \eqref{KBM}, cf. \eqref{ss1}, yields 
\begin{align*}
\frac{1}{8\sqrt{s}}
\exp\left(-\frac{x^{2}\pi}{16s}\right)
&\sum_{n=-\infty}^{\infty}\exp\left(-\frac{n^{2}\pi}{4s}\right)(-1)^n\cosh\left(\frac{nx\pi}{4s}\right)\\
= &
\frac{1}{2}
\sum_{n=0}^{\infty}\left(-1\right)^{n}\exp\left(-\left(2n+1
\right)^{2}
\pi s \right)\sin\left(\left(2n+1\right)\frac{\pi\left(x+1\right)}{2}\right).
\end{align*}
Since
$$
\sin\left(\left(2n+1\right)\frac{\pi\left(x+1\right)}{2}\right)=
(-1)^n\cos\left((2n+1)\frac{\pi x}{2}\right)
$$
it is seen that (\ref{eq:21}) holds when $z=\pi x/2$ and $t=4s$.
\end{proof}
\begin{remark} 
Putting $z=0$ in (\ref{eq:21}) yields the identity 
\begin{equation}
\label{243}
\sqrt{t}\sum_{n\in\mathbb{Z}}\e^{-\pi t\left(n+\frac{1}{2}\right)^{2}}=\sum_{n\in\mathbb{Z}}\left(-1\right)^{n}\e^{-\frac{\pi n^{2}}{t}},\ t>0,
\end{equation}
and with $t\to\frac{1}{t}$ 
\begin{equation}
\sqrt{t}\sum_{n\in\mathbb{Z}}\left(-1\right)^{n}\e^{-\pi n^{2}t}=\sum_{n\in\mathbb{Z}}\e^{-\frac{\pi}{t}\left(n+\frac{1}{2}\right)^{2}}.\label{eq:42}
\end{equation}
\end{remark} 

\subsection{Third modular identity }
\label{sec34}\medskip 
We prove here that the third modular identity can be obtained from the spectral representation formula for reflected Brownian motion,
\begin{theorem}
The spectral representation formula   \eqref{RBM} implies  the third modular identity   
\eqref{eq:31}.
\end{theorem}
\begin{proof}

Substituting $y=-1$ on  the l.h.s. of (\ref{RBM}) yields  \begin{equation}
\label{lh1}
\frac{1}{2}\frac{1}{\sqrt{2\pi t}}\sum_{n=-\infty}^{\infty}\left(\e^{-\frac{\left(x+1+4n\right)^{2}}{2t}}+\e^{-\frac{\left(x+4n+1\right)^{2}}{2t}}\right)=\frac{1}{\sqrt{2\pi t}}\sum_{n=-\infty}^{\infty}\e^{-\frac{\left(x+1+4n\right)^{2}}{2t}},
\end{equation}
and for the r.h.s 
\begin{align}
\label{lh2}
\hskip-1cm\frac{1}{4}+\frac{1}{2}\sum_{n=1}^{\infty}\e^{-\frac{n^{2}\pi^{2}}{8}t}&\cos\left(\frac{n\pi\left(x+1\right)}{2}\right)\\
&
\nonumber
\hskip1cm=\frac{1}{4}\sum_{n=-\infty}^{\infty}\e^{-\frac{n^{2}\pi^{2}}{8}t}\cos\left(\frac{n\pi\left(x+1\right)}{2}\right).
\end{align}
Substituting in (\ref{lh1}) and   (\ref{lh2}) $t={8s}/{\pi}$ and $x=\frac{4z}{\pi}-1$ ($\Leftrightarrow z=\frac{\pi(x+1)}{4}$) yields the identity
\begin{align}
\label{lh3}
\hskip-1cm\frac{1}{4\sqrt{s}}\exp{\left(-\frac{z^{2}}{\pi s}\right)}\sum_{n=-\infty}^{\infty}&\exp{\left(-\frac{\pi n^{2}}{s}\right)}\exp{\left(-\frac{2nz}{s}\right)}\\
\nonumber 
&\hskip1cm=\frac{1}{4}\sum_{n=-\infty}^{\infty}\exp{\left(-\pi n^{2}s\right)}\cos\left(2nz\right),
\end{align}
which is equivalent with (\ref{eq:31}).
\end{proof}
\begin{remark}1. Putting $z=0$ in (\ref{lh3}) gives perhaps the most familiar  Jacobi identity
\begin{equation}
\sqrt{t}\sum_{n=-\infty}^{\infty}
\e^{-\pi n^{2}t}=\sum_{n=-\infty}^{\infty}
\e^{-\frac{\pi n^{2}}{t}}.\label{eq:32}
\end{equation}
{2.} Substituting in (\ref{lh3}) $z=\pi r$ and $s=\pi v$ yields a slightly different form of (\ref{eq:31})
\begin{equation}
\label{lh10}
\frac{1}{\sqrt{\pi v}}\sum_{n=-\infty}^{\infty}\e^{-\frac{\left(r+n\right)^{2}}{v}}=\sum_{n=-\infty}^{\infty}\e^{-n^{2}\pi^{2}v}\cos\left(2n\pi r\right),
\end{equation}
cf. Biane et al. \cite{BianePitmanYor} formula (2.1). 
\end{remark}

\subsection{Fourth modular identity }
\label{sec36}\medskip 
Finally, we consider the fourth modular identity. 
\begin{theorem}
The spectral representation formula   \eqref{RBM} implies the fourth modular identity   
\eqref{eq:41}.
\end{theorem}
\begin{proof}

Take $y=1$ in (\ref{RBM}) to obtain  
\begin{align}
\label{lh11}
\frac{1}{2}\frac{1}{\sqrt{2\pi t}}\sum_{n=-\infty}^{\infty}&\Big(\e^{-\frac{\left(x+4n-1\right)^{2}}{2t}}+\e^{-\frac{\left(x+4n+3\right)^{2}}{2t}}\Big)
\\
&
\nonumber
\hskip1cm=\frac{1}{4}\sum_{n=-\infty}^{\infty}\e^{-\frac{n^{2}\pi^{2}}{8}t}(-1)^n\cos\left(\frac{n\pi\left(x+1\right)}{2}\right).
\end{align}
Substituting here $z=\pi(x+1)/4$ and $t=8s/\pi$ leads, after some manipulations, to the claimed formula (\ref{eq:41}) -- we skip the details.
\end{proof}
\begin{remark}
\label{r44}
As pointed out earlier, (\ref{eq:41}) is also obtained from (\ref{eq:21}) when $t=1/s$ and $z=\imath x/s$.
\end{remark}

\bigskip
\section{Mittag-Leffler expansions via Brownian motion}
\label{sec41}
In this section,  using the identities in Theorem \ref{t11}, we deduce some well known series expansions of hyperbolic functions, cf . \cite{Gradshteyn} 1.421 p. 36. Traditionally such expansions are obtained by applying  Mittag-Leffler's theorem \cite[p.187]{Ahlfors}.
\begin{theorem} 
\label{MiLe}
The following expansions hold:
\begin{align}
\label{ml1}
&\coth z
=\frac 1 z + 2z\sum_{n=1}^{\infty}\frac{1}{z^{2}+n^{2}\pi^{2}}
=z\sum_{n=-\infty}^{\infty}\frac{1}{z^{2}+n^{2}\pi^{2}},\quad z\not= 0,\\
\label{ml2}
&\frac{1}{\sinh z}
=\frac{1}{z}+2z\sum_{n=1}^{\infty}\frac{\left(-1\right)^{n}}{z^{2}+n^{2}\pi^{2}}
=z\sum_{n=-\infty}^{\infty}\frac{\left(-1\right)^{n}}{z^{2}+n^{2}\pi^{2}},\quad z\not= 0,\\
\label{ml3}
&\tanh(z)=8z\sum_{n=1}^\infty \frac{1}{4z^2+(2n-1)^{2}\pi^{2}},\\
\label{ml4}
&\frac{1}{\cosh\left({z}\right)}
=
4\pi\sum_{n=0}^{\infty}\left(-1\right)^{n}\frac{2n+1}
{4 z^2+(2n+1)^2\pi^2}.
\end{align}
\end{theorem}

\begin{proof}
The first expansion (\ref{ml1}) results, after simple manipulations,  from  (\ref{RBMG2}) by taking $x=y=1$. Also the second one (\ref{ml2})  comes from (\ref{RBMG2})  by taking now $x=1, y=-1$. The third one (\ref{ml3}) is deduced from   (\ref{KBMG2}) 
when $x=y=0$. Finally, we prove (\ref{ml4}). For this recall that $Y$ denotes a Brownian motion killed when it exits $(-1,+1)$
and define  
 \begin{align}
\label{HH}
 H:=\inf\left\{ t:Y_{t}\protect\notin\left(-1,1\right)\right\}.
 \end{align}
Then it holds
\begin{align}
\label{exit1}
\P\left\{ H>t\,|\, Y_0=x\right\}   &=\int_{-1}^{1}p_{Y}\left(t;x,y\right)2dy.
\end{align}
For $\alpha>0,$ let $T_\alpha$ be an exponentially (with mean $1/\alpha$) distributed random variable independent of $Y$. Using the expression for $G^{Y}_\alpha$ given on the left-hand side of (\ref{KBMG2}) yields  
\begin{align}
\label{cosh1}
\P\left\{ H>T_\alpha\,|\, Y_0=0\right\}  & =\int_0^\infty dt\, \alpha\, \e^{-\alpha t} \int_{-1}^{1} dy\, 2\,p_{Y}\left(t;x,y\right)
\\
\nonumber
 & =\int_{-1}^{1}  2\alpha\, G^{Y}_\alpha(0,y)\, dy
 \\
 \nonumber
&=1-\frac 1{\cosh(\sqrt{2\alpha})}.
\end{align}
On the other hand, applying the expression for $G^{Y}_\alpha$ given on the right-hand side of (\ref{KBMG2}) produces 
\begin{align}
\label{cosh2}
\P\left\{ H>T_\alpha\,|\, Y_0=0\right\}  & =\int_{-1}^{1}  2\alpha\, G^{Y}_\alpha(0,y)\, dy
 \\
 \nonumber
&
=\sum_{n=0}^{\infty}\left(-1\right)^{n}\, \frac {4}{(2n+1)\pi}\, \frac{8\alpha}
{8\alpha+(2n+1)^2\pi^2}.
\end{align}
Identity (\ref{ml4}) now follows when equating  (\ref{cosh1}) and (\ref{cosh2}).
\end{proof}

\section{Modular identites and hitting times}
\label{51}   

Let $H$ the random variable defined in (\ref{HH}) and let $f_H$ denote its density. Using in (\ref{exit1}) the two different expressions for $p_Y$ as given in Theorem \ref{t1} yields  the modular identity stated in the next proposition. As explained in Remark \ref{conn}, this identity is  related to identity (\ref{eq:theta1 modular}). 

\begin{theorem}
\label{prop51}
The following identity holds
\begin{align}
\label{mod4}
\frac{\pi}{4}\sum_{n=-\infty}^{\infty}\left(-1\right)^{n}\left(2n+1\right)&\exp{\lp -\frac{\left(2n+1\right)^{2}\pi^{2}}{8}t\rp}\\
\nonumber
=\frac{1}{\sqrt{2\pi t^{3}}}&\sum_{n=-\infty}^{\infty}\left(-1\right)^{n}\left(2n+1\right)\exp\lp{-\frac{\left(2n+1\right)^{2}}{2t}}\rp.
\end{align}
\end{theorem}

\begin{proof}
Consider identity (\ref{exit1}). Using the expression for  $p_Y$ as given on the right-hand side of (\ref{KBM}) produces, after some manipulations (and use of Fubini's theorem),
$$
\P\left\{ H>t\,|\, Y_0=0\right\} 
=\sum_{n=0}^{\infty}\left(-1\right)^{n}\frac{4}{\left(2n+1\right)\pi}\e^{-\frac{\left(2n+1\right)^{2}\pi^{2}}{8}t}.
$$
Consequently,  the probability density of $H$ is
\begin{align}
\label{exit3}
\hskip-.5cm f_{H}\left(t\right):=
-\frac{\partial}{\partial t}\P\left\{ H>t\,|\, Y_0=0\right\} &=\frac{\pi}{2}\sum_{n=0}^{\infty}\left(-1\right)^{n}\left(2n+1\right)\e^{-\frac{\left(2n+1\right)^{2}\pi^{2}}{8}t},
\end{align}
which coincides with the left-hand side of (\ref{mod4}). On the other hand, we also have 
\begin{align}
\label{exit2}
f_H(t)  &=-\int_{-1}^{1}\frac{\partial}{\partial t} p_{Y}\left(t;0,y\right)2dy
=-\int_{-1}^{1}\frac{\partial^2}{\partial y^2}p_{Y}\left(t;0,y\right)\, dy\\
\nonumber
&=-\frac{\partial}{\partial y}p_{Y}\left(t;0,1\right)+\frac{\partial}{\partial y}p_{Y}\left(t;0,-1\right),
\end{align}
where we have exchanged the order of differentiation and integration, and used the fact that $p_Y$ solves the heat equation 
$$ 
\frac{\partial}{\partial t} p_{Y}\left(t;x,y\right)=\frac 12\frac{\partial^2}{\partial y^2}p_{Y}\left(t;x,y\right).
$$
Applying in (\ref{exit2}) for $p_Y$ the expression on the left-hand side of (\ref{KBM}) produces the formula on the right-hand side of (\ref{mod4}). 
\end{proof}

\begin{remark}
\label{conn}
The identity (\ref{mod4}) is also a consequence of the modular identity (\ref{eq:theta1 modular}). To see this, let  
\[
\theta_{1}'\left(z\vert\tau\right):=\frac{\partial}{\partial z}\theta_{1}\left(z\vert\tau\right).
\]
Differentiating  in (\ref{eq:theta1 modular}) with respect to $z$ yields 
\[
\left(-\imath\tau\right)^{\frac{1}{2}}\theta_{1}'\left(z\vert\tau\right)=-\imath \e^{\imath\tau'\frac{z^{2}}{\pi}}\left[\imath\tau'\frac{2z}{\pi}\theta_{1}\left(z\tau'\vert\tau'\right)+\tau'\theta_{1}'\left(z\tau'\vert\tau'\right)\right].
\]
Evaluating at $z=0$ produces
\[
\left(-\imath\tau\right)^{\frac{1}{2}}\theta_{1}'\left(0\vert\tau\right)=-\imath\tau'\theta_{1}'\left(0\vert\tau'\right)=\frac{-1}{\imath\tau}\theta_{1}'\left(0\vert\tau'\right)
\]
so that
\begin{equation}
\label{new111}
\left(-\imath\tau\right)^{\frac{3}{2}}\theta_{1}'\left(0\vert\tau\right)=\theta_{1}'\left(0\vert\tau'\right),
\end{equation}
which can be seen to be equivalent with (\ref{mod4}).
\end{remark}


Next we consider  a \textit{3-dimensional Bessel process} $R=\left(R_{t}\right)_{t\ge0}$ initiated at 0. Define the first exit time from $[0,1)$  associated with $R$ via
\[
H_{1}=\inf\left\{ t:R_{t}=1\right\}. 
\]
We focus on two different ways to invert the Laplace transform of $H_1$ yielding  two different expressions for the density and the distribution function of $H_1$. This leads to  an identity which is a special case of the modular identity (\ref{eq:theta2 modular}) displayed in (\ref{eq:42}).

\begin{theorem}
\label{prop52}
The following identity holds
\begin{align}
\label{CDF}
2\sum_{n=1}^{\infty}\left(-1\right)^{n+1}\frac{n^{2}\pi^{2}}{2}\,&\exp\lp{-\frac{n^{2}\pi^{2}}{2}t}\rp\\
\nonumber
&=\frac{{2}}{t^2\sqrt{2\pi t}}\sum_{n=0}^{\infty}\,\left(\left(2n+1\right)^{2}-t\right)\,\exp\lp{-\frac{\left(2n+1\right)^{2}}{2t}}\rp.
\end{align} 
\end{theorem}

\begin{proof}
Recall \cite{BorodinSalminen} formula  (5.2.0.2) on p. 469,
\begin{equation}
\label{LT1}
\mathbb{E}\left(\e^{-\alpha H_{1}}\,|\, R_0=0\right) 
=\frac{\sqrt{2\alpha }}{\sinh\left(\sqrt{2\alpha}\right)}.
\end{equation}
To invert this Laplace transform we apply first the Mittag-Leffler expansion  (\ref{ml2}) which gives 
\begin{equation}
\nonumber 
\frac{\sqrt{2\alpha}}{\sinh\left(\sqrt{2\alpha}\right)}
=1+2\alpha\sum_{n=1}^{\infty}\frac{\left(-1\right)^{n}}{\alpha+\frac{n^{2}\pi^{2}}{2}}. 
\end{equation}
Using next well-known calculation and inversion rules for the Laplace transform, we deduce for $t>0$ the following expression for the density $f_{H_1}$ of  $H_{1}$
\begin{align}
\label{eq:B}
f_{H_1}(t)&=2\sum_{n=1}^{\infty}\left(-1\right)^{n+1}\frac{n^{2}\pi^{2}}{2}\,\e^{-\frac{n^{2}\pi^{2}}{2}t}
=\sum_{n=-\infty}^{\infty}\left(-1\right)^{n+1}\frac{n^{2}\pi^{2}}{2}\,\e^{-\frac{n^{2}\pi^{2}}{2}t},
\end{align}
 which  coincides with the formula (4.8.0.2) on p. 404  
 in \cite{BorodinSalminen} when substituting therein $\nu=1/2$.  
 Another way to invert the Laplace transform in (\ref{LT1}) is to use the series expansion 
 \begin{align*}
\frac{\sqrt{2\alpha}}{\sinh\left(\sqrt{2\alpha}\right)}
=2\sqrt{2\alpha}\,\frac{\e^{-\sqrt{2\alpha}}}{1-\e^{-2\sqrt{2\alpha}}}
&=4\alpha\,\sum_{n=0}^\infty \,\frac{1}{\sqrt{2\alpha}}\,\e^{-(2n+1)\sqrt{2\alpha}}.
\end{align*}
Again, by well-known formulas for the Laplace transform, we obtain for  $t>0$
  \begin{align}
 \label{eq:D}
f_{H_1}(t)&=
4\,\sum_{n=0}^\infty \,
\frac{\partial}{\partial t}\left(
\frac{1}{\sqrt{2\pi t}}\,
\e^{-\frac{\left(2n+1\right)^{2}}{2t}}\right)
\\
\nonumber
&
=\frac{2}{t^2\sqrt{2\pi t}}\sum_{n=0}^{\infty}\,\left(\left(2n+1\right)^{2}-t\right)\e^{-\frac{\left(2n+1\right)^{2}}{2t}},
 \end{align}
 cf. \cite{BorodinSalminen} identity (5.2.0.2) p. 469. 
The identity (\ref{CDF}) follows from (\ref{eq:B}) and (\ref{eq:D}). 
\end{proof}

\begin{remark}
\label{RMK}
From (\ref{CDF}) we may deduce (\ref{eq:42}). Indeed,  apply Fubini's theorem on the left-hand side of (\ref{CDF}) yields
\[
\mathbb{P}\left(H_{1}\ge t\,|\, R_0=0\right)=2\sum_{n=1}^{\infty}\left(-1\right)^{n+1}\e^{-\frac{n^{2}\pi^{2}}{2}t}
\]
and, consequently,
\begin{align}
\mathbb{P}\left(H_{1}\le t\,|\, R_0=0\right) & =1-2\sum_{n=1}^{\infty}\left(-1\right)^{n+1}\e^{-\frac{n^{2}\pi^{2}}{2}t}\label{eq:C}
 =1+2\sum_{n=1}^{\infty}\left(-1\right)^{n}\e^{-\frac{n^{2}\pi^{2}}{2}t}. 
\end{align}
Fubini's theorem can also be used on the right-hand side of (\ref{CDF}) when integrating up to $t$. Hence 
\begin{align}
\label{CDF2}
\mathbb{P}\left(H_{1}\le t\,|\, R_0=0\right)
&= 4\sum_{n=0}^{\infty}\,\frac{{1}}{\sqrt{2\pi t}}\,\e^{-\frac{\left(2n+1\right)^{2}}{2t}}
= 2\sum_{n=-\infty}^{\infty}\,\frac{{1}}{\sqrt{2\pi t}}\,\e^{-\frac{\left(2n+1\right)^{2}}{2t}}.
\end{align}
Comparing (\ref{eq:C}) and (\ref{CDF2}) we deduce the identity 
\begin{align}
\label{CDF3}
1+2\sum_{n=1}^{\infty}\left(-1\right)^{n}\exp\lp{-\frac{n^{2}\pi^{2}}{2}t}\rp &=2\sum_{n=-\infty}^{\infty}\,\frac{{1}}{\sqrt{2\pi t}}\,\exp\lp{-\frac{\left(2n+1\right)^{2}}{2t}}\rp
\end{align}
which is equivalent with (\ref{eq:42}). Of course, differentiating in (\ref{CDF3}) with respect to $t$ yields (\ref{CDF}). 
\end{remark}

\medskip
\section{Discrete Gaussian distributions and duality}
\label{sec4}
\subsection{Discrete Gaussian distributions}
Since the Jacobi theta functions appear naturally in the study of various Brownian motions, we take now a closer look at these functions seen as probability distributions. We begin with by introducing two discrete probability distributions on 
$\mathbb{Z}$ associated with Jacobi theta functions  $\theta_{2}$ and $\theta_{3}$.  

\begin{definition}\label{23theta}\vskip2mm
\noindent
{\bf 1.} An integer valued  random variable $X_{\theta_{2}}$ is said to follow
the discrete Gaussian $\theta_{2}$-distribution  if 
\begin{equation}
\label{2theta}
\P\left\{ X_{\theta_{2}}=n\right\} =\frac{\e^{-c\pi\left(n+\frac{1}{2}\right)^{2}}}{\theta_2(0, \e^{-c\pi})},\qquad n\in\mathbb{Z},\,\,c>0,
\end{equation}
where  the $\theta_2$-function is given in Section \ref{sec1} and $ \tau = \imath c$ is called the \textit{lattice parameter} of the $\theta_{2}$-distribution. 

\noindent
{\bf 2.}  An integer valued  random variable $X_{\theta_{3}}$ is said to have
the discrete Gaussian $\theta_{3}$-distribution  with lattice parameter $\tau = \imath c$ if 
\begin{equation}
\label{3theta}
\P\left\{ X_{\theta_{3}}=n\right\} =\frac{\e^{-c\pi n^{2}}}{\theta_{3}\left(0,\e^{-c\pi}\right)},\qquad  n\in\mathbb{Z},\,\,c>0,
\end{equation}
where  the $\theta_3$-function is given in Section \ref{sec1}. 
\end{definition}

\noindent
From these definitions we deduce immediately the following elementary results. In particular, notice that $X_{\theta_2}+\frac 12$ is a symmetric random variable around 0.

\begin{theorem}
 \label{means}The mean values of $X_{\theta_2}$ and $X_{\theta_3}$ are given by 
 \begin{equation}
 \E(X_{\theta_2})=-1/2\qquad {\text{and}} \qquad  \E(X_{\theta_3})=0.
\end{equation}
\end{theorem}

Other parametrizations of the $\theta_2$- and $\theta_3$-distributions are via elliptic integrals. To introduce these, recall the definitions of Legendre's elliptic integrals (see \cite{NIST} pp. 486-487) 
$$
K(k) :=\int_0^1\frac {dt}{\sqrt{1-t^2}\sqrt{1-k^2t^2}},\qquad E(k) :=\int_0^1\frac {\sqrt{1-k^2t^2}}{\sqrt{1-t^2}}dt
$$
with \textit{the elliptic modulus} $k\in (0,1).$ The \textit{complementary modulus} is defined by $k'=\sqrt{1-k^2}$ and a standard notation is $K'\left(k\right):=K\left(k'\right)$ and $E'\left(k\right):=E\left(k'\right)$.
\begin{lemma}
\label{BIJ}
The function 
$
f\left(k\right):={K(k')}/{K\left(k\right)},
$
is decreasing over  $k \in \left(0,1\right).$ Moreover,
$
\lim_{k\to 0}f(k)=+\infty\quad {\text{and}}\quad \lim_{k\to 1}f(k)=0.
$
\end{lemma}
\begin{proof}
The proof is based on two identities. The first one expresses the derivative of the elliptic integral $K\left(k\right)$ as
\[
\frac{d}{dk}K\left(k\right)=\frac{1}{k}\left[\frac{E\left(k\right)}{k'^{2}}-K\left(k\right)\right].
\]
The second is the famous Legendre relation \cite[19.7.1]{NIST}
\begin{equation}
K\left(k\right)E\left(k'\right)+K(k')E\left(k\right)=\frac{\pi}{2}+K\left(k\right)K(k').
\label{Legendre}
\end{equation}
Additional elementary algebra shows that the derivative of $f(k)$ is indeed negative.
Moreover, the values 
\[
K\left(0\right) = \frac{\pi}{2},\quad K\left(1\right) = +\infty 
\]
produce the desired limits.
\end{proof}
\noindent
Since the change of variables 
\[
k\mapsto\frac{K(k')}{K\left(k\right)}
\]
is, by Proposition \ref{BIJ}, a bijection from $\left(0,1\right)$ to $\mathbb{R}^{+}$,
the equation 
\begin{equation}
\label{EKV}
c={K(k')}/{K\left(k\right)}
\end{equation}
has, for every $c>0,$ a unique
solution $k=k(c)\in\left(0,1\right)$. Using the standard notation for the \textit{nome} 
\[q=\e^{-\pi\frac{K(k')}{K\left(k\right)}},\]
Jacobi's identity \cite[Entry 6 p.101]{Ramanujan},
\[
\theta_{3}\left(0,\e^{-\pi\frac{K(k')}{K\left(k\right)}}\right)={\sqrt{\frac{2}{\pi}K\left(k\right)}}
\]
provides an equivalent but more natural parameterization of the discrete
Gaussian $\theta_{3}$-distribution by the elliptic modulus $k\in\left(0,1\right)$
under the form 
\begin{equation}
\label{31theta}
\P\left\{ X_{\theta_{3}}=n\right\} =\frac{1}{\sqrt{\frac{2}{\pi}K\left(k\right)}}\e^{-\pi\frac{K(k')}{K\left(k\right)}n^{2}},\thinspace\thinspace n\in\mathbb{Z}.  
\end{equation}
Similarly to the Gaussian $\theta_{3}$-distribution, a more natural parameterization
of the $\theta_2$-distribution by the elliptic modulus $k\in\left(0,1\right)$ is  deduced from Jacobi's identity as
\begin{equation}
\label{21theta}
\P\left\{ X_{\theta_{2}}=n\right\} =\frac{1}{\sqrt{\frac{2}{\pi}kK\left(k\right)}}\e^{-\pi\frac{K(k')}{K\left(k\right)}\left(n+\frac{1}{2}\right)^{2}},\thinspace\thinspace n\in\mathbb{Z}.
 \end{equation}
 
Let $\sigma^2(\theta)$ denote the variance of a discrete Gaussian $\theta$-distribution with $\theta=\theta_2$ or $\theta_3$. If needed, the parameter of the distribution is included in the notation as $\sigma^2(\theta(c))$ or $\sigma^2(\theta(k))$.    The variances of the $\theta$-distributions given in the next theorem are derived in \cite{WakhareVignat}.
 \begin{theorem}
The variance of the  ${\theta_{2}}$-distribution is given by 
\begin{equation}
\sigma^2\left({\theta_{2}}\right)=\frac{1}{\pi^{2}}E\left(k\right)K\left(k\right)\label{eq:var2}
\end{equation}
and is equivalently expressed as the Lambert series (with $q=\e^{-\pi\frac{K(k')}{K\left(k\right)}}$)
\begin{equation}
\sigma^2\left({\theta_{2}}\right)=\frac{1}{4}+2\sum_{n\ge1}\frac{\left(-1\right)^{n-1}nq^{2n}}{1-q^{2n}}=\frac{1}{4}+2\sum_{n\ge1}\frac{q^{2n}}{\left(1+q^{2n}\right)^{2}}.
\label{eq:var2_lambert}
\end{equation}
For the $\theta_3$ distribution it holds
\begin{equation}
    \sigma^2\left({\theta_{3}}\right)=\E\lp X_{\theta_{3}}^2\rp= \frac{1}{\pi^{2}}E\left(k\right)K\left(k\right)-\frac{1}{4}\theta_{4}^{4}\left(0,q\right)
    =\frac{K\left(k\right)^{2}}{\pi^{2}}\left[\frac{E\left(k\right)}{K\left(k\right)}-\left(k'\right)^{2}\right],\label{eq:var3}
\end{equation}
and, as Lambert series,
\begin{equation}
\sigma^2\left({\theta_{3}}\right)=2\sum_{n \ge 1}\frac{\left(-1\right)^{n-1}nq^{n}}{1-q^{2n}}=2\sum_{n \ge 1}\frac{q^{2n-1}}{\left(1+q^{2n-1}\right)^{2}}.
\label{eq:var3_lambert}
\end{equation}
\end{theorem}
The cumulants $\kappa_n(\theta)$ of the theta distributions were computed in \cite{WakhareVignat} and also appear as (20.6.7) and (20.6.8) in \cite{NIST}. They are as follows with, in \eqref{cumulants_theta_2}, the Euler polynomials $E_n\left(x\right)$ defined by the generating function
\[
\sum_{n=0}^{\infty}
\frac{E_n\left(x\right)}{n!}z^n = \frac{2e^{xz}}{1+e^{z}},\,\,\vert z \vert < \pi.
\]
\begin{theorem}
The even-indexed cumulants of the $\theta_2$ and $\theta_3$ distributions have the Lambert series representations
\begin{equation}
\kappa_{2n}\left({\theta_{2}\left(c\right)}\right)=-\frac{1}{2}E_{2k-1}\left(0\right)+\sum_{m\ge1}\left(-1\right)^{m+1}m^{2n-1}\frac{\e^{-cm\pi}}{\sinh\left(cm\pi\right)},\thinspace\thinspace n\ge1,
\label{cumulants_theta_2}
\end{equation}
and
\begin{equation}
\kappa_{2n}\left({\theta_{3}\left(c\right)}\right)=\sum_{m\ge1}\frac{\left(-1\right)^{m+1}m^{2n-1}}{\sinh\left(cm\pi\right)},\thinspace\thinspace n\ge1,
\label{cumulants_theta_3}
\end{equation} 
and the Eisenstein series representations
\begin{align}
\kappa_{2n}\left({\theta_{2}\left(c\right)}\right) & =\frac{\left(-1\right)^{n+1}}{\pi^{2n}}\left(2n-1\right)!\sum_{n_{1},n_{2}\in\mathbb{Z}}\frac{1}{\left(\left(2n_{1}-1\right)+\imath c\left(2n_{2}\right)\right)^{2n}}
\label{cumulants_theta_2_Eisenstein}
\end{align}
and
\begin{align}
\kappa_{2n}\left({\theta_{3}\left(c\right)}\right) & =\frac{\left(-1\right)^{n+1}}{\pi^{2n}}\left(2n-1\right)!\sum_{n_{1},n_{2}\in\mathbb{Z}}\frac{1}{\left(\left(2n_{1}-1\right)+\imath c\left(2n_{2}-1\right)\right)^{2n}}.
\label{cumulants_theta_3_Eisenstein}
\end{align}
The odd-indexed cumulants satisfy
\begin{equation}
\kappa_{2n+1}\left({\theta_{2}}\right)=
\begin{cases}
1/2,& n=0,\\ 
0,& n\ge 1,
\end{cases}
\end{equation}
and 
\begin{equation}
\kappa_{2n+1}\left({\theta_{3}}\right)=0,\thinspace\thinspace n\ge0.
\end{equation}
\end{theorem}

\begin{remark}
 \label{oddcumulant}
 Since $X_{\theta_2}+\frac 12$ is symmetric, the fact that $\kappa_{2n+1}\left({\theta_{2}}\right)=0$ for $n\ge 1$ can also be deduced from the translation invariance of the cumulants.  
\end{remark}

While the variance of a discrete Gaussian distribution seems difficult to evaluate using 
\eqref{eq:var2}, 
\eqref{eq:var2_lambert},
\eqref{eq:var3} and 
\eqref{eq:var3_lambert},
closed-form expressions exist in some particular cases. These expressions appeal
to the theory of \textit{singular values} of the elliptic integral, i.e. values $k_{r}$ of the elliptic modulus that satisfy the equation
\[
\frac{K\left(k_{r}'\right)}{K\left(k_{r}\right)}=\sqrt{r},
\]
where $r$ is a rational number. Ramanujan computed the explicit values of
$102$ of these elliptic moduli called \textit{singular moduli}, see \cite{Borwein}.
 
Selberg and Chowla \cite{Selberg} showed that the values $K\left(k_{r}\right)$
can be expressed in terms of a finite number
of the Gamma function (see Table \ref{table1} for some examples). The elliptic alpha function $\alpha\left(r\right),$ that takes simple algebraic values (see Table \ref{table1}) such as $\alpha\left(1\right)=1,\,\,\alpha\left(2\right)=\sqrt{2}-1 \dots$
is then used to relate the values of  $E\left(k_{r}\right)$
to the values of $K\left(k_{r}\right)$ according to
\[
\alpha\left(r\right)=\frac{\pi}{4K\left(k_{r}\right)^{2}}+\sqrt{r}-\sqrt{r}\frac{E\left(k_{r}\right)}{K\left(k_{r}\right)}.
\]
In terms of the elliptic alpha function $\alpha\left(r\right)$ and the complete elliptic integrals $K\left(k_{r}\right)$, the variance \eqref{eq:var3} of the $\theta_{3}$ distribution
 reads
 \begin{equation}
 \sigma^{2}\left(\theta_{3}\right)
 \label{eq:sigmaXk}
 =\frac{1}{4\pi\sqrt{r}}+\frac{K^{2}\left(k_{r}\right)}{\pi^{2}}\left(k_{r}^{2}-\frac{\alpha\left(r\right)}{\sqrt{r}}\right)
 \nonumber 
 \end{equation}
and the variance \eqref{eq:var2} of the $\theta_{2}$ distribution reads
 \begin{equation}
 \sigma^{2}\left(\theta_{2}\right)
 =\frac{1}{4\pi\sqrt{r}}+\frac{K^{2}\left(k_{r}\right)}{\pi^{2}}\left(1-\frac{\alpha\left(r\right)}{\sqrt{r}}\right).
 \nonumber 
 \end{equation}
The explicit values of 
the variances
$\sigma^{2}\left(\theta_{2}\right)$ and $\sigma^{2}\left(\theta_{3}\right)$ for several values
of the parameter $r$ are given in Table \ref{table2} and Table \ref{table3}, respectively, in the Appendix.

\subsection{Stochastic representations for the $\theta_{2}$- and $\theta_{3}$-random variables}
This subsection introduces stochastic representations for the $\theta_{2}$ and $\theta_{3}$ random variables as series of independent Bernoulli random variables.
\begin{theorem}
With $c>0,$ consider two independent sequences of independent Bernoulli random
variables $\left\{ Z_{1}^{-},Z_{2}^{-},\dots\right\} $ and $\left\{ Z_{1}^{+},Z_{2}^{+},\dots\right\} $
such that, for $n \ge 1,$
\[
\P\left\{ Z_{n}^{+}=+\frac{1}{2}\right\} =\P\left\{ Z_{n}^{-}=-\frac{1}{2}\right\} =\frac{\e^{\left(n-\frac{1}{2}\right)\pi c}}{2\cosh\left(\left(n-\frac{1}{2}\right)\pi c\right)}
\]
and
\[
\P\left\{ Z_{n}^{+}=-\frac{1}{2}\right\} =\P\left\{ Z_{n}^{-}=+\frac{1}{2}\right\} =\frac{\e^{-\left(n-\frac{1}{2}\right)\pi c}}{2\cosh\left(\left(n-\frac{1}{2}\right)\pi c\right)}.
\]
Then 
\begin{equation}
X_{\theta_{3}\left(c\right)}\rlaw\sum_{n\ge1}\lp Z_{n}^{+}+Z_{n}^{-}\rp,\label{eq:theta3}
\end{equation}
where the $\rlaw$ sign indicates equality in distribution.
Consider, moreover, two independent sequences of independent Bernoulli random
variables $\left\{ Y_{0}^{-},Y_{1}^{-},\dots\right\} $ and $\left\{ Y_{0}^{+},Y_{1}^{+},\dots\right\} $
such that, for $n \ge 0$
\[
\P\left\{ Y_{n}^{+}=+\frac{1}{2}\right\} =\P\left\{ Y_{n}^{-}=-\frac{1}{2}\right\} =\frac{\e^{n\pi c}}{2\cosh\left(n\pi c\right)}
\]
and
\[
\P\left\{ Y_{n}^{+}=-\frac{1}{2}\right\} =\P\left\{ Y_{n}^{-}=+\frac{1}{2}\right\} =\frac{\e^{-n\pi c}}{2\cosh\left(n\pi c\right)}.
\]
Then 
\begin{equation}
X_{\theta_{2}\left(c\right)}\rlaw-\frac{1}{2}+Y_{0}^{+}+\sum_{n\ge1}\lp Y_{n}^{+}+Y_{n}^{-}\rp.\label{eq:theta2}
\end{equation}
\end{theorem}
\begin{proof}
These expressions are deduced from the infinite product representations for the normalized $\theta_2$- and $\theta_3$-functions that appear as formulas (20.5.7) and (20.5.6) in \cite{NIST},
respectively
\[
\frac{\theta_{3}\left(z,\tau\right)}{\theta_{3}\left(0,\tau\right)}=\prod_{n\ge1}\frac{\cos\left(\left(n-\frac{1}{2}\right)\pi\tau+z\right)\cos\left(\left(n-\frac{1}{2}\right)\pi\tau-z\right)}{\cos^{2}\left(\left(n-\frac{1}{2}\right)\pi\tau\right)}
\]
and
\[
\frac{\theta_{2}\left(z,\tau\right)}{\theta_{2}\left(0,\tau\right)}=\cos z\prod_{n\ge1}\frac{\cos\left(n\pi\tau+z\right)\cos\left(n\pi\tau-z\right)}{\cos^{2}\left(n\pi\tau\right)}
\]
with the lattice parameter $\tau=\imath c.$ These identities are direct consequences of Jacobi's triple product identity. Identifying with the moment generating functions
for $X_{\theta_{2}\left(c\right)}$ and $X_{\theta_{3}\left(c\right)},$
respectively, that is, 
\begin{equation}
\label{MOM1}
 \mathbb{E}\lp e^{zX_{\theta_{2}\left(c\right)}}\rp =\e^{-\frac{z}{2}}\frac{\theta_{2}\left(\frac{z}{2\imath},\tau\right)}{\theta_{2}\left(0,\tau\right)},\thinspace\thinspace{\text{ and}} \ \ \mathbb{E}\lp\e^{zX_{\theta_{3}\left(c\right)}}\rp=\frac{\theta_{3}\left(\frac{z}{2\imath},\tau\right)}{\theta_{3}\left(0,\tau\right)},   
\end{equation}
produces the result.
\end{proof}
We conclude this subsection with two remarks.
\begin{remark}
In \cite{Kemp}, a $\theta_3$-distributed random variable is characterized as
the difference of two independent Heine distributed random variables:
\[
X_{\theta_3(c)} = X_{A} - X_{B}
\]
with
\[
\P\left({X_{A}}=i\right) = \P\left({X_{B}}=i\right) = \P\left({X_{A}}=0\right) \frac{q^{i^2}}{\left(q^2;q^2\right)_i},\,\,i=0,1,2\dots
\]
and with the q-Pochhammer symbol $\left(a;q\right)_{i} = \prod_{k=0}^{i-1}\left(1-aq^k\right).$
These Heine random variables $X_{A}$  and $X_{B}$ 
can now be identified as 
\[
X_{A}=\sum_{n\ge1}Z_{n}^{+},\thinspace\thinspace
X_{B}=-\sum_{n\ge1}Z_{n}^{-}
\]
 in the stochastic representation \eqref{eq:theta3}.
\end{remark}
\begin{remark}
As an application of the stochastic representation \eqref{eq:theta3}, the cumulants of $X_{\theta_{3}}$ can be computed as follows:
\begin{align*}
\kappa_{2m}\left(\theta_{3}\left(c\right) \right) = \sum_{n\ge 1}\kappa_{2m}\left(Z_{n}^{+} + Z_{n}^{-}\right) =  \sum_{n\ge 1}\kappa_{2m}\left(Z_{n}^{+}\right)
+ \sum_{n\ge 1}\kappa_{2m}\left(Z_{n}^{-}\right),
\end{align*}
where we have used the additive property of the cumulants. Explicit computation of the elementary cumulants $\kappa_{2m}\left(Z_{n}^{+}\right)$ and $\kappa_{2m}\left(Z_{n}^{-}\right)$ and elementary algebra produces
\begin{align*}
\kappa_{2m}\left({\theta_{3}\left(c\right)}\right) &  =2\sum_{n\ge1}\left(-q^{2n-1}\right)\frac{A_{2m-1}\left(-q^{2n-1}\right)}{\left(1+q^{2n-1}\right)^{2m}},
\end{align*}
with $A_m\left(x\right)$ the Eulerian polynomials that
satisfy the identity 
\[
\frac{xA_{m-1}\left(x\right)}{\left(1-x\right)^{m}}=\sum_{n\ge1}n^{m-1}x^{n}
\]
so that, with $x=-q^{2n-1}$ and after simple algebra, we find
\begin{align*}
\kappa_{2m}\left({\theta_{3}\left(c\right)}\right) &  =2\sum_{n\ge1}\left(-1\right)^{n}n^{2m-1}\frac{q^{n}}{1-q^{2n}}=\sum_{n\ge1}\left(-1\right)^{n-1}\frac{n^{2m-1}}{\sinh\left(cn\pi\right)},\thinspace\thinspace m\ge1,
\end{align*}
that coincides with \eqref{cumulants_theta_3}.
The same approach using \eqref{eq:theta2} produces the cumulants \eqref{cumulants_theta_2}.
\end{remark}

\subsection{Duality}
In this subsection, we study the consequences of some modular properties of the Jacobi $\theta$-functions on the corresponding $\theta$-distributed random variables. 
We consider first the $\theta_3$-distribution. Our starting point is the third modular identity as given in (\ref{lh10}) with $v=c/\pi$, i.e.  
\begin{equation}
\label{lh100}
\frac{1}{\sqrt{c}}\sum_{n=-\infty}^{\infty}\e^{-\frac{\pi\left(r+n\right)^{2}}{c}}=\sum_{n=-\infty}^{\infty}\e^{-c\pi n^{2}}\cos\left(2n\pi r\right).
\end{equation}
Expanding both sides of (\ref{lh100}) in MacLaurin's series in $r$  and comparing the coefficients yields the following duality result. We denote with $\mu_n(\theta)$ the $m$th centered moment of a $\theta$-distribution. 
\begin{theorem}
\label{DU1}
Consider two  $\theta_3$-distributed random variables $X_{k}$ and $X_{k'}$ with respective elliptic moduli $k$ and $k'=\sqrt{1-k^2}$
or, equivalently, with respective lattice parameters $\imath c$ and $\frac{\imath}{c}$
defined by \eqref{EKV}.  Then
\begin{align}
\label{REC}
\frac{\mu_{2m}(\theta_3(k))}{K(k)^{2m}}&- (-1)^m\frac{\mu_{2m}(\theta_3(k'))}{K(k')^{2m}}\\
\nonumber
&=\sum_{l=1}^m \frac{(2m)!}{l!\,(2m-2l))!}\frac{(-1)^{m-l}}{(4\pi K(k)K(k'))^l}
\frac{\mu_{2m-2l}(\theta_3(k'))}{K(k')^{2(m-l)}}.
\end{align}
\end{theorem}

\begin{corollary}
\label{DU2}
Let $m=1$ in (\ref{REC}). Then
\begin{align}
\label{REC1}
\frac{\mu_{2}(\theta_3(k))}{K(k)^{2}}+
\frac{\mu_{2}(\theta_3(k'))}{K(k')^{2}}
&=\frac{1}{2\pi K(k)K(k')}.
\end{align}
\end{corollary}

\begin{remark}
 Identity (\ref{REC1}) is given in \cite[(17)]{WakhareVignat} p. 300. Therein it is deduced from (\ref{eq:var3}) via Legendre's identity \eqref{Legendre}. Notice also that (\ref{REC1}) can be expressed using the lattice parameter $\imath c=\imath\frac{K(k')}{K\left(k\right)}$ as
\begin{equation}
 c\,\sigma^2\lp\theta_3(c)\rp +\frac{1}{c}\sigma^2\lp\theta_3({1}/{c})\rp=\frac{1}{2\pi}.
\end{equation}
\end{remark}

\begin{corollary}
\label{DU3}
Let $c=1$ (or equivalently $k={1}/{\sqrt{2}}$) in (\ref{REC}). Then
\begin{align}
\label{REC2}
\lp 1-(-1)^m\rp& \frac{\mu_{2m}(\theta_3(k))}{K(k)^{2m}}\\
\nonumber 
&=\sum_{l=1}^m \frac{(2m)!}{l!\,(2(m-l))!}\frac{(-1)^{m-l}}{(4\pi K(k)^2)^l}
\frac{\mu_{2m-2l}(\theta_3(k))}{K(k)^{2(m-l)}}.
\end{align}
\end{corollary}

\begin{corollary}
\label{DU4}
Let $c=1$ and $m=1$ in (\ref{REC}). Then (see first row in Table \ref{table3})
\begin{align}
\label{REC3}
\sigma^2(\theta_3(1))
= \frac 1{4\pi}.
\end{align}
\end{corollary}

\begin{theorem}
\label{PR0}
Let $X_{\theta_2}$ and $X_{\theta_3}$ be parameterized by the elliptic modulus $k$. Then 
 \begin{align}
\label{ODD}
\E\lp (-1)^{X_{\theta_2}}\rp =0\qquad {\text{and}}\qquad 
\E\lp (-1)^{X_{\theta_3}}\rp =\sqrt{ k'}.
\end{align}
\end{theorem}
\begin{proof}
 The claim concerning the $\theta_2$-distribution is immediate from (\ref{2theta}). For the second statement  we use the modular identity (\ref{lh100}). Indeed, putting therein $r=\frac{1}{2}$ yields
$$
 \frac{1}{\sqrt{c}}\sum_{n=-\infty}^{\infty}\e^{-{\pi\left(n+\frac 12\right)^{2}}/{c}}=\sum_{n=-\infty}^{\infty}(-1)^n\e^{-c\pi n^{2}}.
$$
Using the normalizations in (\ref{31theta}) and (\ref{21theta}) this identity is equivalent to 
$$
\E\lp (-1)^{X_{\theta_3}}\rp\,\sqrt{\frac2\pi K(k)}
= \frac{1}{\sqrt{c}}\sqrt{\frac2\pi k'K(k')},
$$
from which, evoking (\ref{EKV}), the claim follows.  
\end{proof}
\noindent
As a consequence of Theorem \ref{PR0}, we have
\begin{corollary}
\label{OE}
Let  $X_{\theta_2}$ and $X_{\theta_3}$ be as in Theorem \ref{PR0}. Then  
$$
 \P\lp X_{\theta_2} {\text { is odd }}\rp=1/2\qquad {\text{and}}\qquad 
\P\lp X_{\theta_3} {\text { is odd }} \rp = (1-\sqrt{ k'})/2.
 $$
\end{corollary}

\begin{theorem}
\label{PR1}
Consider two  $\theta_2$-distributed random variables $X_{k}$ and $X_{k'}$ with respective elliptic moduli $k$ and $k'=\sqrt{1-k^2}$
(equivalently with respective lattice parameters $\imath c$ and $\frac{1}{\imath c}$
defined by (\ref{EKV})).
Then
\begin{align}
\label{REC2}
\frac{\sqrt{k}}{K(k)}{\E\Big((-1)^{X_{k}}\,X_{k}\Big) }
=
\frac{\sqrt{k'}}{K(k')}{\E\lp (-1)^{ X_{k'}}\,
X_{k'}\rp }.
\end{align} 
 \end{theorem}
 \begin{proof}
The claim follows from the modular identity (\ref{mod4}) with calculations similar as in the proof of Proposition \ref{PR0}. We skip the details.    
 \end{proof}
\begin{remark}
In \cite{Kemp}, the discrete normal  $\theta_3$-distribution is characterized
as the discrete distribution on $\mathbb{Z}$
with maximum Shannon entropy given a fixed variance. The entropy of the $\theta_3$-distribution with parameter $k$ is easily computed as 
\[
h\left(X_{\theta_3}\right)=-\sum p_{n}\log p_{n}=\frac{1}{2}\log\left(\frac{2}{\pi}K\left(k\right)\right)+\pi K\left(k\right)K'\left(k\right)\frac{\sigma_{X_{k}}^{2}}{K^{2}\left(k\right)}
\]
Assuming that $X_{k}$ and $X_{k'}$ are two $\theta_{3}$-  distributed random variables with respective parameters $k$ and $k'=\sqrt{1-k^2}$, the sum of their Shannon entropies is
\[
h\left(X_{k}\right) + h\left(X_{k'}\right) = \frac{1}{2}+\frac{1}{2}\log\left(\frac{4}{\pi^{2}}K\left(k\right)K\left(k'\right)\right).
\]
Elementary algebra shows that this is a convex function of $k$ with  its minimal value
\[
\underset{k\in\left(0,1\right)}{\min} \left(h\left(X_{k}\right) + h\left(X_{k'}\right)\right) = 
\frac{1}{2} + \frac{1}{2}\log \left(\frac{\pi }{\Gamma \left(\frac{3}{4}\right)^4}\right)
\]
being reached in the lemniscatic case $k=k'=\frac{1}{\sqrt{2}}$.
\end{remark}
 We conclude this subsection by considering the Kolmogorov distribution. 
 As an introduction, let 
$(W^{0,1,0}_s)_{0\leq s\leq 1}$ denote 
a Brownian bridge from $0$ to $0$ of length $1$, see \cite{BorodinSalminen}. Then for $h>0$ 
\begin{align}
\label{Br2}
\P\left\{\sup_{s< 1}|W^{0,1,0}_s|<h\right\}   
&=\sum_{n=-\infty}^{\infty}(-1)^n \e^{-{2n^2h^2}}=:F(h).
\end{align}
The distribution function $h\mapsto F(h), h>0,$ is called the Kolmogorov distribution function, see \cite{kolmogorov33}, and also \cite{smirnov39}. For a derivation of (\ref{Br2}) in the spirit of Section 5, see \cite{GeissLuotoSalminen}. For an approach based on the excursion theory, see \cite{PermanWellner}. 

\begin{theorem}
\label{KOL}    
Let $X_{\theta_3}$ be a ${\theta_3}$-distributed random variable with lattice parameter $\imath c=2 \imath h^2/\pi=K(k')/K(k)$. Then the Kolmogorov distribution function $F$ and its density $F'$ have the following representations
$$
F(h)= \sqrt{\frac2\pi K(k)}\,\mathbb{E}\lp (-1)^{X_{\theta_3}}\rp=\sqrt{\frac2\pi k'K(k)}
$$
and 
\begin{align*}
F'\left(h\right) & 
=-4\sqrt{K(k')}\,\mathbb{E}\left(\left(-1\right)^{X_{\theta_3}}X_{\theta_3}^{2}\right)\\
&=\frac 1{\pi^2}\frac{K(k)}{K(k')}\frac 1{\sqrt{K(k')}} \lp \frac 12 \sqrt{k'K(k')}-\frac 1{\sqrt{2\pi}}\, E(k')K(k)\rp
\end{align*}
\end{theorem}

\begin{proof}
The representation for $F$ follows from the definition of the $\theta_3$-distribution combined with Proposition \ref{PR0}. For the density $F'$ it holds
$$
F'\left(h\right)  =-4h\sum_{n\in\mathbb{Z}}\left(-1\right)^{n}n^{2}\e^{-2n^{2}h^{2}},
$$
and this can be developed, via the modular identity (\ref{CDF}) (see also (\ref{CDF3})), to an expression containing the first two moments of the $\theta_2$- distribution. Applying then (\ref{eq:var2}) yields the claimed formula for $F'$. We skip the details.
\end{proof}



\subsection{Landen transform and an addition formula for $\theta_2$- and $\theta_3$-distributed random variables }
In this subsection, we consider another classic transformation of the nome $q$, the Landen transformation
$q^2\mapsto q$. One of the identities associated to this transformation is as follows \cite[20.7.12]{NIST}:
\[
\frac{\theta_{2}\left(z,q^2\right)}{\theta_{2}\left(0,q^2\right)}
\frac{\theta_{3}\left(z,q^2\right)}{\theta_{3}\left(0,q^2\right)}=\frac{\theta_{2}\left(z,q\right)}{\theta_{2}\left(0,q\right)}.
\]
For an interpretation
of this identity in terms of lattice properties, see \cite[Example 4.A.3]{Sebestyen}.
The transformation of the nome $q^2\mapsto q$ is equivalent to a transformation of the lattice parameter $\imath c \mapsto \imath \tilde{c}$ with
\[
\tilde{c}=\frac{c}{2}
\]
and to the transformation of the elliptic parameter $k \mapsto \tilde{k}$
as
\[
\tilde{k}=\frac{2\sqrt{k}}{1+k}
\]
and corresponds to the \textit{ascending Landen transformation}. The \textit{descending Landen transformation} corresponds to the inverse transformation $q\mapsto q^2$ of the nome.

Using probabilistic tools related to the discrete normal distribution only, we produce four different proofs of the following result. The second proof is based on the Landen transformation.
\begin{theorem}
If $X_{\theta_{2}\left(c\right)}$ and $X_{\theta_{3}\left(c\right)}$
are independent and  $\theta_{2}\left(c\right)$- and $\theta_3\left(c\right)$-distributed, respectively, random variables then
\begin{equation}
X_{\theta_{2}\left(c\right)
}+X_{\theta_{3}\left(c\right)}
\rlaw X_{\theta_{2}\left(\frac{c}{2}\right)}.\label{eq:convolution}
\end{equation}
\end{theorem}
\subsubsection*{First proof, using the probability distribution} With $c=\frac{K(k')}{K\left(k\right)},$ the
distribution of the random variable $Y=X_{\theta_{2}\left(c\right)}+X_{\theta_{3}\left(c\right)}$
is computed as the convolution
\begin{align*}
\P\left\{ Y=m\right\}  & =\frac{1}{\sqrt{\frac{2}{\pi}K\left(k\right)}}\frac{1}{\sqrt{\frac{2}{\pi}kK\left(k\right)}}\sum_{n\in\mathbb{Z}}\e^{-\pi c\left(m-n\right)^{2}}\e^{-\pi c\left(n+\frac{1}{2}\right)^{2}}\\
 & =\frac{\pi}{2\sqrt{k}K\left(k\right)}\sum_{n\in\mathbb{Z}}\e^{-\pi c\left(m+\frac{1}{2}-\left(n+\frac{1}{2}\right)\right)^{2}}\e^{-\pi c\left(n+\frac{1}{2}\right)^{2}}\\
 & =\frac{\pi}{2\sqrt{k}K\left(k\right)}\e^{-\pi c\left(m+\frac{1}{2}\right)^{2}}\sum_{n\in\mathbb{Z}}\e^{2\pi c\left(m+\frac{1}{2}\right)\left(n+\frac{1}{2}\right)}\e^{-2\pi c\left(n+\frac{1}{2}\right)^{2}}.
\end{align*}
Completing the square in the right-hand side exponential produces
\begin{align*}
\P\left\{ Y=k\right\}  & =\frac{\pi}{2\sqrt{k}K\left(k\right)}\e^{-\pi c\left(m+\frac{1}{2}\right)^{2}}\sum_{n\in\mathbb{Z}}\e^{-\pi c\left[2\left(n+\frac{1}{2}\right)^{2}-2\left(m+\frac{1}{2}\right)\left(n+\frac{1}{2}\right)+\frac{\left(m+\frac{1}{2}\right)^{2}}{2}\right]}\e^{\pi c\frac{\left(m+\frac{1}{2}\right)^{2}}{2}}\\
 & =\frac{\pi}{2\sqrt{k}K\left(k\right)}\e^{-\pi c\frac{\left(m+\frac{1}{2}\right)^{2}}{2}}\sum_{n\in\mathbb{Z}}\e^{-\pi c\left(\sqrt{2}\left(n+\frac{1}{2}\right)-\frac{\left(m+\frac{1}{2}\right)}{\sqrt{2}}\right)^{2}}\\
 & =\frac{\pi}{2\sqrt{k}K\left(k\right)}\e^{-\pi c\frac{\left(m+\frac{1}{2}\right)^{2}}{2}}\sum_{n\in\mathbb{Z}}\e^{-2\pi c\left(\left(n+\frac{1}{2}\right)-\frac{\left(m+\frac{1}{2}\right)}{2}\right)^{2}}\\
 & =\frac{\pi}{2\sqrt{k}K\left(k\right)}\e^{-\pi c\frac{\left(m+\frac{1}{2}\right)^{2}}{2}}\sum_{n\in\mathbb{Z}}\e^{-2\pi c\left(n-\frac{m}{2}+\frac{1}{4}\right)^{2}}.
\end{align*}
If $m$ is even, $m=2p$ then
\[
\left(n-\frac{m}{2}+\frac{1}{4}\right)^{2}=\left(n-p+\frac{1}{4}\right)^{2}
\]
while if $m$ is odd, $m=2p+1$ then
\[
\left(n-m+\frac{1}{4}\right)^{2}=\left(n-p-\frac{1}{4}\right)^{2}=\left(-\left(n-p\right)+\frac{1}{4}\right)^{2}
\]
so that in both cases
\[
\sum_{n\in\mathbb{Z}}\e^{-2\pi c\left(n-\frac{m}{2}+\frac{1}{4}\right)^{2}}=\sum_{n\in\mathbb{Z}}\e^{-2\pi c\left(n+\frac{1}{4}\right)^{2}}
\]
does not depend on $m.$
\subsubsection*{Second proof, using moment generating functions}
The moment generating functions, cf. (\ref{MOM1}), are, with $q=\e^{-\pi c},$
\[
\mathbb{E}\lp\e^{zX_{\theta_{2}\left(c\right)}}\rp=\e^{-\frac{z}{2}}\frac{\theta_{2}\left(\frac{z}{2\imath},q\right)}{\theta_{2}\left(0,q\right)},\thinspace\thinspace\mathbb{E}\lp\e^{zX_{\theta_{3}\left(c\right)}}\rp=\frac{\theta_{3}\left(\frac{z}{2\imath},q\right)}{\theta_{3}\left(0,q\right)}
\]
and identity 20.7.12 in \cite{NIST} produces
\[
\frac{\theta_{2}\left(z,q^{2}\right)\theta_{3}\left(z,q^{2}\right)}{\theta_{2}\left(z,q\right)}=\frac{1}{2}\theta_{2}\left(0,q\right)
\]
so that
\[
\theta_{2}\left(z,q^{2}\right)\theta_{3}\left(z,q^{2}\right)=\frac{1}{2}\theta_{2}\left(z,q\right)\theta_{2}\left(0,q\right)
\]
and evaluating at $z=0$
\[
\theta_{2}\left(0,q^{2}\right)\theta_{3}\left(0,q^{2}\right)=\frac{1}{2}\theta_{2}^{2}\left(0,q\right)
\]
so that
\[
\frac{\e^{-\frac{z}{2}}\theta_{2}\left(z,q^{2}\right)\theta_{3}\left(z,q^{2}\right)}{\theta_{2}\left(0,q^{2}\right)\theta_{3}\left(0,q^{2}\right)}=\frac{\e^{-\frac{z}{2}}\theta_{2}\left(z,q\right)}{\theta_{2}\left(0,q\right)}.
\]
Replacing $q^{2}$ by $q$ and $z$ by $\frac{z}{2\imath}$ produces
the result.
\subsubsection*{Third proof, using cumulants and their Lambert series representation}
From the expressions for the cumulants \eqref{cumulants_theta_2} and \eqref{cumulants_theta_3}, we deduce, for $n\ge1,$
\begin{align*}
\kappa_{2n}\left({\theta_{2}\left(c\right)}+{\theta_{3}\left(c\right)}\right) & =-\frac{1}{2}E_{2k-1}\left(0\right)+\sum_{k\ge1}\frac{\left(-1\right)^{k+1}k^{2n-1}}{\sinh\left(ck\pi\right)}\left(1+\e^{-ck\pi}\right)\\
 & =-\frac{1}{2}E_{2k-1}\left(0\right)+\sum_{k\ge1}\frac{\left(-1\right)^{k+1}k^{2n-1}}{2\cosh\left(\frac{ck\pi}{2}\right)\sinh\left(\frac{ck\pi}{2}\right)}\e^{-\frac{ck\pi}{2}}2\cosh\left(\frac{ck\pi}{2}\right)\\
 & =-\frac{1}{2}E_{2k-1}\left(0\right)+\sum_{k\ge1}\frac{\left(-1\right)^{k+1}k^{2n-1}}{\sinh\left(\frac{ck\pi}{2}\right)}\e^{-\frac{ck\pi}{2}}=\kappa_{2n}\left(X_{\theta_{2}\left(\frac{c}{2}\right)}\right).
\end{align*}
and 
\[
\kappa_{2n+1}\left({\theta_{2}\left(c\right)}+{\theta_{3}\left(c\right)}\right)=0,\thinspace\thinspace n\ge0,
\]
where $\kappa_{n}\left({\theta_{2}\left(c\right)}+{\theta_{3}\left(c\right)}\right), n\ge 1,$ refers the cumulants of 
$X_{\theta_{2}\left(c\right)}+X_{\theta_{3}\left(c\right)}$.
\subsubsection*{Fourth proof, using cumulants and their Eisenstein series representation}

Using the Eisenstein series representations \eqref{cumulants_theta_2_Eisenstein} and \eqref{cumulants_theta_3_Eisenstein} of the non-zero cumulants,
we deduce
\begin{align*}
\kappa_{2n}\left({\theta_{3}\left(c\right)}+{\theta_{2}\left(c\right)}\right) & =\kappa_{2n}\left({\theta_{3}\left(c\right)}\right)+\kappa_{2n}\left({\theta_{2}\left(c\right)}\right)\\
 & =\frac{\left(-1\right)^{n+1}}{\pi^{2n}}\left(2n-1\right)!\sum_{n_{1}\ge1,n_{2}\in\mathbb{Z}}\frac{1}{\left(\left(2n_{1}-1\right)+\imath c\left(2n_{2}-1\right)\right)^{2n}}\\
 & +\frac{\left(-1\right)^{n+1}}{\pi^{2n}}\left(2n-1\right)!\sum_{n_{1},n_{2}\in\mathbb{Z}}\frac{1}{\left(\left(2n_{1}-1\right)+\imath c\left(2n_{2}\right)\right)^{2n}}\\
 & =\frac{\left(-1\right)^{n+1}}{\pi^{2n}}\left(2n-1\right)!\sum_{n_{1},n_{2}\in\mathbb{Z}}\frac{1}{\left(\left(2n_{1}-1\right)+\imath\frac{c}{2}\left(2n_{2}\right)\right)^{2n}}\\
 & =\kappa_{2n}\left({\theta_{2}\left(\frac{c}{2}\right)}\right).
\end{align*}

\subsection{Stability results}
The {\em continuous} Gaussian distribution enjoys the famous stability property: if $X$ and $Y$ are independent continuous Gaussian random variables with respective variances $\sigma_{X}^{2}$ and $\sigma_{Y}^{2}$ then their sum $X+Y$ is Gaussian with variance $\sigma_{X}^{2}+\sigma_{Y}^{2}.$
A natural question at this point is whether this property extends to the case of discrete Gaussian random variables. Although the answer is negative, we can produce the following result.
\begin{theorem}
Assume that $X_{\theta_{3}\left(c\right)}$ and $X'_{\theta_{3}\left(c\right)}$
are two independent $\theta_{3}\left(c\right)-$distributed random variables. Then for two arbitrary real numbers   $a$ and $b$ it holds that
\begin{equation}
aX_{\theta_{3}\left(c\right)}+bX_{\theta_{3}\left(c\right)}^{'}
\rlaw
\pi_c\,Z_3  + (1-\pi_c)\, Z_2 
\label{eq:theta3conv}
\end{equation}
where
$$
Z_3:=a\left(X_{\theta_{3}\left(2c\right)}+X_{\theta_{3}\left(2c\right)}^{'}\right)+b\left(X_{\theta_{3}\left(2c\right)}-X_{\theta_{3}\left(2c\right)}^{'}\right),
$$

$$
Z_2:=a\left(X_{\theta_{2}\left(2c\right)}+X_{\theta_{2}\left(2c\right)}^{'}+1\right)+b\left(X_{\theta_{2}\left(2c\right)}-X_{\theta_{2}\left(2c\right)}^{'}\right)
$$
and $\pi_c$ is a Bernoulli random variable 
with
\begin{equation}
\P(\pi_c=1)=p_{c}=\frac{\theta_{3}^{2}\left(0,q^{2}\right)}{\theta_{3}^{2}\left(0,q\right)}=\frac{\theta_{3}^{2}\left(0,\e^{-2\pi c}\right)}{\theta_{3}^{2}\left(0,\e^{-\pi c}\right)}=\frac{1+\sqrt{1-k^{2}}}{2}=1-k\left(q^{2}\right)
\label{p_c}
\end{equation}
and 
\begin{equation}
\P(\pi_c=0)=1-p_{c}=\frac{\theta_{2}^{2}\left(0,q^{2}\right)}{\theta_{3}^{2}\left(0,q\right)}=\frac{1-\sqrt{1-k^{2}}}{2}=k\left(q^{2}\right)\label{1-p_c}.
\end{equation}
The random variables  $ X_{\theta_{2}\left(2c\right)}, X_{\theta_{2}\left(2c\right)}^{'}, X_{\theta_{3}\left(2c\right)}, X_{\theta_{3}\left(2c\right)}^{'},$ and $\pi_c$ are taken to be independent.
\end{theorem}
\begin{proof}
This result is deduced from the identity \cite[p.8]{Lawden}
\begin{equation}
\theta_{3}\left(x,q\right)\theta_{3}\left(y,q\right)=\theta_{3}\left(x+y,q^{2}\right)\theta_{3}\left(x-y,q^{2}\right)+\theta_{2}\left(x+y,q^{2}\right)\theta_{2}\left(x-y,q^{2}\right)
\label{Lawden19}    
\end{equation}

so that
\begin{align*}
\frac{\theta_{3}\left(x,q\right)}{\theta_{3}\left(0,q\right)}\frac{\theta_{3}\left(y,q\right)}{\theta_{3}\left(0,q\right)} & =\frac{\theta_{3}\left(x+y,q^{2}\right)}{\theta_{3}\left(0,q^{2}\right)}\frac{\theta_{3}\left(x-y,q^{2}\right)}{\theta_{3}\left(0,q^{2}\right)}\frac{\theta_{3}^{2}\left(0,q^{2}\right)}{\theta_{3}^{2}\left(0,q\right)}\\
 & +\frac{\theta_{2}\left(x+y,q^{2}\right)}{\theta_{2}\left(0,q^{2}\right)}\frac{\theta_{2}\left(x-y,q^{2}\right)}{\theta_{2}\left(0,q^{2}\right)}\frac{\theta_{2}^{2}\left(0,q^{2}\right)}{\theta_{3}^{2}\left(0,q\right)}
\end{align*}
and
\begin{align*}
\mathbb{E}\lp \e^{xX_{\theta_{3}\left(c\right)}+yX'_{\theta_{3}\left(c\right)}}\rp & =p_{c}\mathbb{E}\lp \e^{\left(x+y\right)X_{\theta_{3}\left(2c\right)}+\left(x-y\right)X'_{\theta_{3}\left(2c\right)}}\rp \\
 & +\left(1-p_{c}\right)\e^{x}\mathbb{E}\lp \e^{\left(x+y\right)X_{\theta_{2}\left(2c\right)}+\left(x-y\right)X_{\theta_{2}\left(2c\right)}^{'}}\rp.
\end{align*}
Substituting $x=az,\thinspace\thinspace y=bz$ produces
\begin{align*}
\mathbb{E}\lp \e^{z\left(aX_{\theta_{3}\left(c\right)}+bX'_{\theta_{3}\left(c\right)}\right)}\rp  & =p_{c}\mathbb{E}\lp \e^{z\left(\left(a+b\right)X_{\theta_{3}\left(2c\right)}+\left(a-b\right)X'_{\theta_{3}\left(2c\right)}\right)}\rp\\
 & +\left(1-p_{c}\right)\e^{az}\mathbb{E}\lp \e^{z\left(\left(a+b\right)X_{\theta_{2}\left(2c\right)}+\left(a-b\right)X_{\theta_{2}\left(2c\right)}^{'}\right)}\rp
\end{align*}
so that finally
\[
aX_{\theta_{3}\left(c\right)}+bX_{\theta_{3}\left(c\right)}^{'}\rlaw\begin{cases}
a\left(X_{\theta_{3}\left(2c\right)}+X_{\theta_{3}\left(2c\right)}^{'}\right)+b\left(X_{\theta_{3}\left(2c\right)}-X_{\theta_{3}\left(2c\right)}^{'}\right) \\
a\left(X_{\theta_{2}\left(2c\right)}+X_{\theta_{2}\left(2c\right)}^{'}+1\right)+b\left(X_{\theta_{2}\left(2c\right)}-X_{\theta_{2}\left(2c\right)}^{'}\right) 
\end{cases}
\]
with respective probabilities $p_{c}$ and $1-p_{c}$ as defined in \eqref{p_c} and \eqref{1-p_c}.
\end{proof}
The special case $a=b=1$ in the previous result produces the following corollary.
\begin{corollary}
If $X_{\theta_{3}\left(c\right)}$ and $X'_{\theta_{3}\left(c\right)}$ are two independent $\theta_{3}\left(c\right)$ random variables, then
\[
X_{\theta_{3}\left(c\right)}+X_{\theta_{3}\left(c\right)}^{'}
\rlaw\begin{cases}
2X_{\theta_{3}\left(2c\right)} 
\\
2X_{\theta_{2}\left(2c\right)}+1 
\end{cases}
\]
with respective probabilities $p_{c}$ and $1-p_{c}$ as defined in \eqref{p_c} and \eqref{1-p_c}.
\end{corollary}
\section{Conclusion}
The main goal of this article was to study the interaction between analytic properties of elliptic functions and probabilistic properties of random processes and variables built from these functions. Several paths of further study emerge from this first approach. 

The first one is related to multivariate Jacobi functions as briefly mentioned in \cite{Bellman}, see also \cite{Agostini} for a detailed study of the multivariate Gaussian distributions and Chapter 21 of \cite{NIST} for their link with Riemann surfaces. Higher-dimensional $\theta$ functions still satisfy modular transformation as a consequence of the higher-dimensional Poisson summation formula (see Section 69 in \cite{Bellman}). It will be interesting to study what interactions are at play between the modular properties of these higher-dimensional elliptic functions and higher-dimensional random processes.

Another path is the connection between Jacobi $\theta$ functions and lattice summation theory: to each $\theta$ random variable is associated an underlying lattice in the complex plane. The convolution law \eqref{eq:theta3conv} of the $\theta_3$ random variable  is a consequence of invariance properties (such as \eqref{Lawden19}) of the $\theta_{3}$ Jacobi function  with respect to the transformations of the underlying lattice. A better understanding of this correspondence is the goal of future studies. 

\section{Acknowledgements}
C. Vignat would like to thank P. Salminen and the Mathematics Department at  Abo Akademi University for their warm welcome over all these years.
\newpage
\section{Appendix}
\begin{table}[htbp]
\begin{centering}
\hspace*{-2.5cm}
\begin{tabular}{|c|c|c|c|}
\hline 
$r$ & $k_{r}$ & $K\left(k_{r}\right)$ & $\alpha\left(r\right)$\\[0.75ex]
\hline 
\hline 
& \\[-2ex]
1 & $\frac{\sqrt{2}}{2}$ & $\frac{\Gamma^{2}\left(\frac{1}{4}\right)}{4\sqrt{\pi}}$ & $\frac{1}{2}$
\\[2ex]
\hline 
& \\[-2ex]  
2 & $\sqrt{2}-1$ & $\frac{\sqrt{\sqrt{2}+1}}{2^{\frac{13}{4}}\sqrt{\pi}}\Gamma\left(\frac{1}{8}\right)\Gamma\left(\frac{3}{8}\right)$ & $\sqrt{2}-1$
\\[2ex]
\hline 
& \\[-2ex] 
3 & $\frac{\sqrt{2}}{4}\left(\sqrt{3}-1\right)$ & $\frac{3^{\frac{1}{4}}}{2^{\frac{7}{3}}\pi}\Gamma^{3}\left(\frac{1}{3}\right)$ & $\frac{\sqrt{3}-1}{2}$
\\[2ex]
\hline 
& \\[-2ex] 
4 & $3-2\sqrt{2}$ & $\frac{\sqrt{2}+1}{2^{\frac{7}{2}}\sqrt{\pi}}\Gamma^{2}\left(\frac{1}{4}\right)$ & $2\left(\sqrt{2}-1\right)^{2}$\\[2ex]
\hline 
& \\[-2ex] 
5 & $\sqrt{\frac{1}{2}-\sqrt{\sqrt{5}-2}}$ & $\sqrt[4]{\sqrt{5}+2}\sqrt{\frac{\Gamma\left(\frac{1}{20}\right)\Gamma\left(\frac{3}{20}\right)\Gamma\left(\frac{7}{20}\right)\Gamma\left(\frac{9}{20}\right)}{160\pi}}$ & $\frac{\sqrt{5}-\sqrt{2\sqrt{5}-2}}{2}$\\[2ex]
\hline 
& \\[-2ex] 
6 & $\left(2-\sqrt{3}\right)\left(\sqrt{3}-\sqrt{2}\right)$ & 
$\begin{array}{c}
\sqrt{\left(\sqrt{2}-1\right)\left(\sqrt{2}+\sqrt{3}\right)\left(\sqrt{3}+2\right)}\\
\times\sqrt{\frac{\Gamma\left(\frac{1}{24}\right)\Gamma\left(\frac{5}{24}\right)\Gamma\left(\frac{7}{24}\right)\Gamma\left(\frac{11}{24}\right)}{384\pi}}
\end{array}$
& $5\sqrt{6}+6\sqrt{3}-8\sqrt{2}-11$\\[4ex]
\hline 
& \\[-2ex]  
7 & $\frac{1}{8}\sqrt{2}\left(3-\sqrt{7}\right)$ & $\frac{\Gamma\left(\frac{1}{7}\right)\Gamma\left(\frac{2}{7}\right)\Gamma\left(\frac{4}{7}\right)}{\left(\sqrt[4]{7}4\right)\pi}$ & $\frac{\sqrt{7}-2}{2}$\\[2ex]
\hline 
& \\[-2ex] 
8 & $\left(\sqrt{2}-\sqrt{2\sqrt{2}+2}+1\right)^{2}$ & $\frac{\sqrt{\frac{2\sqrt{2}+\sqrt{5\sqrt{2}+1}}{4\sqrt{2}}}\sqrt[4]{\sqrt{2}+1}\Gamma\left(\frac{1}{8}\right)\Gamma\left(\frac{3}{8}\right)}{8\sqrt{\pi}}$ & $2\left(7\sqrt{2}+10\right)\left(1-\sqrt{\sqrt{8}-2}\right)^{2}$\\[2ex]
\hline 
& \\[-2ex] 
9 & $\frac{1}{2}\left(\sqrt{2}-\sqrt[4]{3}\right)\left(\sqrt{3}-1\right)$ & $\frac{\sqrt[4]{3}\sqrt{\sqrt{3}+2}\Gamma\left(\frac{1}{4}\right)^{2}}{12\sqrt{\pi}}$ & $\frac{1}{2}\left(3-3^{3/4}\sqrt{2}\left(\sqrt{3}-1\right)\right)$
\\[2ex]
\hline 
& \\[-2ex] 
10 & $\left(\sqrt{10}-3\right)\left(\sqrt{2}-1\right)^{2}$ &
$\begin{array}{c}
\sqrt{3\sqrt{2}+\sqrt{5}+2}\\
\times \sqrt{\frac{\Gamma\left(\frac{1}{40}\right)\Gamma\left(\frac{7}{40}\right)\Gamma\left(\frac{9}{40}\right)\Gamma\left(\frac{11}{40}\right)\Gamma\left(\frac{13}{40}\right)\Gamma\left(\frac{19}{40}\right)\Gamma\left(\frac{23}{40}\right)\Gamma\left(\frac{37}{40}\right)}{2560\pi^{3}}}
\end{array}$
&
$72\sqrt{2}-46\sqrt{5}+33\sqrt{10}-103$\\[3ex]
\hline 
\end{tabular}
\par\end{centering}
    \caption{Values of the elliptic modulus $k_{r}$, the elliptic integral $K\left(k_{r}\right)$ and the elliptic alpha function $\alpha\left(r\right)$ for $1\le r \le 10 $}
    \label{table1}
\end{table}

\begin{table}[H]
\hspace*{-1.0cm}
\begin{tabular}{|c|c|c|}
\hline
r & $\sigma^2\left( \theta_{2}\right)$ & numerical value
\\[0.75ex] 
\hline 
\hline 
& \\[-2ex]
1 & $\frac{8 \pi ^2+\Gamma \left(\frac{1}{4}\right)^4}{32 \pi ^3}$ & 0.253728
\\[2ex]
\hline 
& \\[-2ex] 
2 & $\frac{32 \pi ^2+\left(\sqrt{2}+2\right) \Gamma \left(\frac{1}{8}\right)^2 \Gamma \left(\frac{3}{8}\right)^2}{128 \sqrt{2} \pi ^3}$ & $0.250277$
\\[2ex]
\hline 
& \\[-2ex]
3 & $\frac{16 \pi ^3+\sqrt[3]{2} \left(\sqrt{3}+3\right) \Gamma \left(\frac{1}{3}\right)^6}{64 \sqrt{3} \pi ^4}$ & $0.250038$
\\[2ex]
\hline 
& \\[-2ex]
4 & $\frac{8 \pi ^2+\left(\sqrt{2}+1\right) \Gamma \left(\frac{1}{4}\right)^4}{64 \pi ^3}$ & $0.250007$
\\[2ex]
\hline 
& \\[-2ex]
5 & $\frac{80 \pi ^2 \sqrt{5}+\left(5 \sqrt{\sqrt{5}+2}+\sqrt{10 \left(\sqrt{5}+3\right)}\right) \Gamma \left(\frac{1}{20}\right) \Gamma \left(\frac{3}{20}\right) \Gamma \left(\frac{7}{20}\right) \Gamma \left(\frac{9}{20}\right)}{1600 \pi ^3}$ & $0.250002$
\\[2ex]
\hline 
& \\[-2ex]
6 & $\frac{96 \pi ^2 \sqrt{6}+\left(6 \sqrt{2}+2 \sqrt{3}+3 \sqrt{6}+6\right) \Gamma \left(\frac{1}{24}\right) \Gamma \left(\frac{5}{24}\right) \Gamma \left(\frac{7}{24}\right) \Gamma \left(\frac{11}{24}\right)}{2304 \pi ^3}$ & $0.25  + 4.1\times 10^{-7}$
\\[2ex]
\hline 
& \\[-2ex]
7 & $\frac{56 \pi ^3+\left(2 \sqrt{7}+7\right) \Gamma \left(\frac{1}{7}\right)^2 \Gamma \left(\frac{2}{7}\right)^2 \Gamma \left(\frac{4}{7}\right)^2}{224 \sqrt{7} \pi ^4}$ & $0.25  + 1.2\times 10^{-7}$\\[2ex]
\hline 
& \\[-2ex]
8 & $\frac{32 \pi ^2+\left(1-\frac{\left(7 \sqrt{2}+10\right) \left(\sqrt{2 \left(\sqrt{2}-1\right)}-1\right)^2}{\sqrt{2}}\right) \sqrt{\sqrt{2}+1} \left(2 \sqrt{2}+\sqrt{5 \sqrt{2}+1}\right) \Gamma \left(\frac{1}{8}\right)^2 \Gamma \left(\frac{3}{8}\right)^2}{256 \sqrt{2} \pi ^3}$ & $0.25 + 3.8\times 10^{-8}$
\\[2ex]
\hline 
& \\[-2ex]
9 & $\frac{24 \pi ^2+\left(\sqrt[4]{3} \sqrt{2}+3^{3/4} \sqrt{2}+2 \sqrt{3}+3\right) \Gamma \left(\frac{1}{4}\right)^4}{288 \pi ^3}$ & $0.25 + 1.3\times 10^{-8}$
\\[2ex]
\hline 
& \\[-2ex]
10 & $\frac{640 \pi ^4 \sqrt{10}+\left(15 \sqrt{2}+10 \sqrt{5}+4 \sqrt{10}+20\right) \Gamma \left(\frac{1}{40}\right) \Gamma \left(\frac{7}{40}\right) \Gamma \left(\frac{9}{40}\right) \Gamma \left(\frac{11}{40}\right) \Gamma \left(\frac{13}{40}\right) \Gamma \left(\frac{19}{40}\right) \Gamma \left(\frac{23}{40}\right) \Gamma \left(\frac{37}{40}\right)}{25600 \pi ^5}$ & $0.25 + 4.7\times 10^{-9}$
\\[2ex]
\hline 
\end{tabular}
    \caption{The theoretical and numerical values of the variance 
$\sigma^{2}\left(\theta_{2}\right)$ for several values of the singular modulus $k_r$}
    \label{table2}
\end{table}
\par

\vspace{1cm}
\hspace{-0.5cm}
\begin{table}[H]
\hspace*{-1.5cm}
\begin{tabular}{|c|c|c|}
\hline
$r$ & $\sigma^{2}\left(\theta_{3}\right)$ & numerical value\\[0.75ex] 
\hline 
\hline 
& \\[-2ex]
1 & $\frac{1}{4\pi}$ & $7.95775\times 10^{-2}$ \\[2ex]
\hline 
& \\[-2ex]
2 & $\frac{32\pi^{2}+\left(\sqrt{2}-2\right)\Gamma\left(\frac{1}{8}\right)^{2}\Gamma\left(\frac{3}{8}\right)^{2}}{128\sqrt{2}\pi^{3}}
$ & $2.29835\times 10^{-2}$
\\[2ex]
\hline 
& \\[-2ex]
3 & $\frac{32\sqrt{3}\pi^{3}-3\sqrt[3]{2}\Gamma\left(\frac{1}{3}\right)^{6}}{384\pi^{4}}
$ & $8.59238\times 10^{-3}$\\[2ex]
\hline 
& \\[-2ex]
4 & $\frac{1}{8\pi}-\frac{4\left(\sqrt{2}-1\right)\Gamma\left(\frac{5}{4}\right)^{4}}{\pi^{3}}
$ & $3.72099\times 10^{-3}$\\[2ex]
\hline 
& \\[-2ex]
5 & $\frac{80\pi^{2}\sqrt{5}+\left(\sqrt{5}-5\right)\Gamma\left(\frac{1}{20}\right)\Gamma\left(\frac{3}{20}\right)\Gamma\left(\frac{7}{20}\right)\Gamma\left(\frac{9}{20}\right)}{1600\pi^{3}}
$ & $1.77591\times 10^{-3}$\\[2ex]
\hline 
& \\[-2ex]
6 & $\frac{96\pi^{2}\sqrt{6}+\left(-6\sqrt{2}+2\sqrt{3}+3\sqrt{6}-6\right)\Gamma\left(\frac{1}{24}\right)\Gamma\left(\frac{5}{24}\right)\Gamma\left(\frac{7}{24}\right)\Gamma\left(\frac{11}{24}\right)}{2304\pi^{3}}
$ & $9.09095\times 10^{-4}$\\[2ex]
\hline 
& \\[-2ex]
7 & $\frac{64\sqrt{7}\pi^{3}-5\Gamma\left(\frac{1}{7}\right)^{2}\Gamma\left(\frac{2}{7}\right)^{2}\Gamma\left(\frac{4}{7}\right)^{2}}{1792\pi^{4}}
$ & $4.90926\times 10^{-4}$\\[2ex]
\hline 
& \\[-2ex]
8 & $\frac{32\pi^{2}+\left(\sqrt{2}-2\sqrt{2\left(\sqrt{2}+1\right)}+2\right)\Gamma\left(\frac{1}{8}\right)^{2}\Gamma\left(\frac{3}{8}\right)^{2}}{256\sqrt{2}\pi^{3}}
$ & $2.76612\times 10^{-4}$\\[2ex]
\hline 
& \\[-2ex]
9 & $\frac{12\pi^{2}-\sqrt{2\sqrt{3}-3}\Gamma\left(\frac{1}{4}\right)^{4}}{144\pi^{3}}
$ & $1.61373\times 10^{-4}$\\[2ex]
\hline 
& \\[-2ex]
10 & $\frac{640\pi^{4}\sqrt{10}+\left(15\sqrt{2}-10\sqrt{5}+4\sqrt{10}-20\right)\Gamma\left(\frac{1}{40}\right)\Gamma\left(\frac{7}{40}\right)\Gamma\left(\frac{9}{40}\right)\Gamma\left(\frac{11}{40}\right)\Gamma\left(\frac{13}{40}\right)\Gamma\left(\frac{19}{40}\right)\Gamma\left(\frac{23}{40}\right)\Gamma\left(\frac{37}{40}\right)}{25600\pi^{5}}
$ & $9.69284\times 10^{-5}$\\[2ex]
\hline 
\end{tabular}
\hspace{1cm}
\caption{The theoretical and numerical values of the variance 
$\sigma^{2}\left(\theta_{3}\right)$ for several values of the singular modulus $k_r$}
\label{table3}
\end{table}
\newpage
\bibliographystyle{amsplain}

\end{document}